\documentclass{article}

\usepackage{amsfonts,amssymb,amsthm,amsmath}
\usepackage[all]{xy}
\usepackage[dvipsnames]{xcolor}

\CompileMatrices

\newcommand{\annihil}{\operatorname{ann}}
\newcommand{\ext}{\operatorname{Ext}}
\renewcommand{\hom}{\operatorname{Hom}}
\newcommand{\der}{\operatorname{Der}}

\newcommand{\adic}{\operatorname{\mathfrak a}}
\newcommand{\exa}{\operatorname{E}}
\newcommand{\rad}{\operatorname{Rad}}
\renewcommand{\Im}{\operatorname{Im}}

\newtheorem{theorem}{Theorem}
\newtheorem{lemma}[theorem]{Lemma}
\newtheorem{proposition}[theorem]{Proposition}
\newtheorem{corollary}[theorem]{Corollary}

\theoremstyle{remark}
\newtheorem{remark}[theorem]{Remark}

\author{Dmitry Trushin}
\title{Algebraization of a Cartier divisor}
\date{}

\begin{document}

\maketitle

\begin{abstract}
We extend to pairs classical results of R. Elkik on lifting of homomorphisms and algebraization. In particular, we establish algebraization of an affine rig-smooth formal variety with a rig-smooth closed subvariety. This solves affirmatively a problem raised by M. Temkin and has applications to desingularization theory.
\end{abstract}

\tableofcontents

\section*{Introduction}

\subsubsection*{Motivation}

In a fundamental work~\cite{ArtAppr}, Artin studied how various algebraic structures over the completion or henselization of a ring $A$ along an ideal $m$ can be approximated with a structure over $A$. Some his results were extended in a fundamental paper of Elkik~\cite{Elkik73}. In particular, Elkik studied approximation and algebraization of rig-smooth schemes and formal schemes defined over a complete or henselian ring. Note also that very recently several constructions of~\cite{Elkik73} were clarified and simplified by Gabber and Ramero in~\cite[Chapter~5]{GabRam}.

The results of~\cite{Elkik73} have various important applications in algebraic geometry and related fields. They are especially useful when one wants to reduce questions about general schemes to the case of varieties. One of such applications is resolution of singularities of quasi-excellent schemes, see~\cite{Temkin2008} and \cite{Temkin2012}. However, the reliance of the latter papers on~\cite{Elkik73} imposed some restrictions due to the fact that a basic datum involved in the strongest desingularization results (often referred to as embedded desingularization) is rather complicated, and its algebraization is not covered by~\cite{Elkik73}. This raised a natural quest for extending the results of~\cite{Elkik73} to more complicated algebraic structures. For example, in the end of~\cite[\S1.1]{Temkin2008} Temkin noted that probably Elkik's algebraization of rig-smooth formal schemes can be extended to pairs consisting of a rig-smooth formal scheme and a rig-snc divisor. In this paper, we explore a slightly different direction. We consider only a rig-smooth formal subscheme but do not restrict its codimension. Our results will be used by Temkin to establish strong embedded desingularization of quasi-excellent schemes of characteristic zero.

\subsubsection*{The method, results, and relation to~\cite{Elkik73}}

Our main goal is to prove Theorem~\ref{theorem:main}, which is a natural generalization of~\cite[Theorem~7]{Elkik73} to rig-smooth pairs. It should be noted that by loc.cit. we can algebraize the scheme and its subscheme independently, and it is also not difficult to obtain a morphism between the algebraizations. The main difficulty is to guarantee that the connecting homomorphism of rings is surjective (i.e. the morphism of schemes is a closed immersion). In order to prove this, we have to generalize almost all essential results of~\cite{Elkik73}. Here is the list of them:
\begin{itemize}
\item Theorem~\ref{theorem:common_alg_lifting} shows the lifting property with values in completions for a pair of rig-smooth algebras. It generalizes Theorem~1.
\item Theorem~\ref{theorem:com_alg_lift_h} shows that every homomorphism of a pair of rig-smooth algebras to a pair of completions can be approximated by a homomorphism to a pair of henselizations. It generalizes Theorem~2~bis.
\item Theorem~\ref{theorem:pair_lift} shows the lifting property with values in henselizations for a pair of rig-smooth algebras. It generalizes Theorem~2 and is just an equivalent form of Theorem~\ref{theorem:com_alg_lift_h}.
\item Theorem~\ref{theorem:duo_mod_algeb} shows that every surjective homomorphism of modules being projective over the complement of $V(\adic)$ is algebraizable. It generalizes Theorem~3.
\item Theorem~\ref{theorem:duo_mod_lift} shows the lifting property for a pair of modules. It generalizes Lemma on page~572.
\item Theorem~\ref{theorem:alg_div} is a corollary of Theorem~\ref{theorem:main}. It shows that every rig-smooth Cartier divisor on a rig-smooth affine formal scheme is algebraizable.
\end{itemize}
We also added Proposition~\ref{prop:app_hens}, whose analogue does not appear in~\cite{Elkik73} but can be found in~\cite[Prop.~3.3.1]{Temkin2012}.

It should be noted that Gabber and Ramero introduce a different Jacobian ideal than the one used in~\cite{Elkik73}. Set-theoretically, it defines the same set of non-smooth points on a scheme but it is more flexible from the scheme-theoretic point of view. In our paper, we use the Gabber-Ramero ideal because it allows us to avoid computation at all and speak on the language of diagrams.

In the text, we have several restrictions. One of them is the number of generators for the defining ideal $\adic$. From the one hand, this restriction is essential only in Proposition~\ref{prop_good_case}, because the original result from~\cite{Elkik73} also has this restriction and we use it in our proof. From the other hand, for the sake of simplification, we intentionally write all other results in the paper in the case of the principal ideal. Using the same induction trick as used in~\cite{Elkik73}, one can easily  extend the results to the general case.

In our statements, we pay attention to different numerical parameters of homomorphisms liftings. One of the most important is the number $h$ such that $\adic^h\subseteq \bar H_{B/A}$. This number is called a conductor in~\cite{Temkin2012}. However, we usually take sums of maximums of conductors in situations when one can use just the sums of them. A careful reading of the proofs may give slightly more general statements than we proved but this is not our main aim.

\subsubsection*{General plan of the proof of Theorem~\ref{theorem:main}}
In a sense, we directly generalize the proof of~\cite[Theorem~7]{Elkik73} with some technical simplifications. We start with a henselian pair $(A,\adic)$, where $\adic=(t)\subseteq A$ is a principal ideal. Then the problem is to algebraize a surjective homomorphism $B\to \bar B$ of formally finitely generated $\widehat{A}$-algebras being formally smooth over the complement of $V(\adic)$. We solve the problem in two steps. The first one is to show the result in a particular case, when we can compute everything explicitly. The second step is to reduce of the general case to this one. The reduction is done by algebraization of a pair of modules being projective over the complement of $V(\adic)$. Now, let us describe the two steps with more details.

\paragraph{Step 1} If we want to algebraize $B=\widehat{A}\{X\}/J$, where $J=(g_1,\ldots,g_k,e)$, we can find some $g_1^0,\ldots,g_k^0,e^0\in A[X]$ such that
$$
g_i \equiv g_i^0 \pmod{\adic^n}\quad \mbox{ and }\quad e\equiv e^0\pmod{\adic^n}.
$$
If $e$ is an idempotent modulo $g_i$ and $g_i$ form a basis of $(J/J^2)_t$ (this determines our special case), we automatically get that $B^0=A[X]/(g_1^0,\ldots,g_k^0,e^0)$ is smooth and, if $n$ was sufficiently large, we automatically derive from this that $B$ and $\widehat{B}^0$ are isomorphic. In order to find such $g_i^0$ and $e^0$, we write down the equation saying that $e$ is an idempotent modulo $g_i$ and solve this equation using lifting Theorem~2 of~\cite{Elkik73}.

If we want to algebraize a pair of algebras given by $J\subseteq \bar J\subseteq \widehat{A}\{X\}$ with $B=\widehat{A}\{X\}/J$ and $\bar B=\widehat{A}\{X\}/\bar J$, we should pose the condition above on both ideals $J$ and $\bar J$ and also add another restriction: we should have an isomorphism of the following form
$$
(J/J^2)\otimes_B \bar B_t = (\bar J/\bar J^2)_t\oplus D
$$
where $D$ is $\bar B_t$-free. Then we can also write down an explicit system of equations on the generators of $J$ and $\bar J$ and find its solution due to Theorem~2 of~\cite{Elkik73}.

\paragraph{Step 2} In order to reduce an arbitrary case to the particular one, we should replace the quotient homomorphism $B\to \bar B$ by homomorphisms of the form $S_{B}\{J/J^2\}\to S_{\bar B}\{\bar J/\bar J^2\}$ and $B\{X\}\to \bar B$. Using separated K\"ahler differentials, one can show that the reduction is always possible. In order to come back to the initial algebras, it is enough to algebraize a pair of modules $P\to \bar P$, where $P_t$ and $\bar P_t$ are projective $B_t$ and $\bar B_t$-modules, respectively (one of the steps uses $P=J/J^2$ and $\bar P = \bar J/\bar J^2$). The case of modules is easer because we can replace, say $P$, by the beginning of its free resolution $B^m\stackrel{L}{\longrightarrow} B^n\to P$. Then we manipulate with finite set of matrices like $L$ and similar ones.

\subsubsection*{Structure of the paper}

In the first section, we fix some terminology and notation. In Section~2, we combine ideas of~\cite{GabRam} with~\cite{Elkik73} and provide scheme-theoretic basis for our needs. Section~3 is devoted to lifting theorems for modules and algebras. This is a technical core of the paper. Approximation of pairs of henselizations and pairs of complete algebras is considered in Section~4. In the final section, we prove our main result that is Theorem~\ref{theorem:main} and its corollary -- Theorem~\ref{theorem:alg_div}.

\subsubsection*{Acknowledgment}

I am grateful to Michael Temkin for suggesting the problem and support during writing the paper. Some parts of the text where significantly reduced due to his suggestions. This research was supported by Marie Curie International Reintegration Grant 268182 within the 7th European Community Framework Programme.

\section{Conventions}

Through out the text we assume that our rings are associative,
commutative and with an identity element. We also suppose that all
our rings are Noetherian. If we are given a ring $A$ with an ideal
$\adic$ and an \'etale homomorphism $A\to B$. We will say that $B$
is strictly \'etale if $A/\adic = B/\adic B$. The pair $(A, \adic)$
is said to be henselian if, for every strictly \'etale homomorphism
$\nu\colon A\to B$, there is a homomorphism $\mu \colon B\to A$ such
that $\mu\circ \nu = Id_A$.

If $A$ is an arbitrary ring and $\adic\subseteq A$ is an arbitrary
ideal, we define the following ring
$$
A^h = \varinjlim\{B\mid A\to B \quad \mbox{ is strictly \'etale}\}
$$
The ring $A^h$ is called the henselization of $A$ with respect to
$\adic$. The completion of a ring $A$ with $\adic$-adic topology
will be denoted by $\widehat{A}$.

For any finite set $x=\{x_1,\ldots,x_n\}$, the rings $A[x]$,
$A[x]^h$, and $A\{x\}$ denote the ring of polynomials, its
henselization with respect to $\adic[x]$, and the completion in
$\adic[x]$-adic topology, respectively. Every derivation
$\partial_i=\partial/\partial x_i$ of the ring $A[x]$ uniquely
extends to derivations in $A[x]^h$ and $A\{x\}$ and its extensions
will be denoted by the same name $\partial_i$. If we are given
elements $f_1,\ldots,f_n\in R$, where $R$ is either $A[x]$,
$A[x]^h$, or $A\{x\}$, then $\triangle^m(f_1,\ldots,f_n)$ will be
the ideal generated by all $m$-minors of the matrix $(\partial_j
f_i)$.

For any ring $A$ the category of $A$-modules will be denoted by
$A{-}\mathrm{mod}$ and the category of finite $A$-modules by
$A{-}\mathrm{mod}^0$.

Let again $A$ be a ring and $\adic\subseteq A$ is an arbitrary
ideal. We supply every $A$-module with $\adic$-adic topology. In
this case, for every $A$-module $M$, every derivation $d\colon A\to
M$ is continuous. Now, assume that $A$ is $\adic$-adically complete.
We will say that an $A$-algebra $B$ is formally finitely generated
if there is a dense finitely generated $A$-algebra. In other words,
$B$ is a quotient of $A\{x\}$ for some finite set $x$. Let us note,
that every formally finitely generated $A$-algebra is complete and
separated by definition. Moreover, every finitely generated module
over $B$ is complete and separated.

Suppose that a formally finitely generated algebra $B$ is presented
as follows
$$
B = A\{x_1,\ldots,x_n\}/J,\quad \mbox{ where }\quad J = (f_1,\ldots,
f_q)
$$
and for some natural $p$, we are given $(\alpha) =
(\alpha_1,\ldots,\alpha_p)$, where
$$
1\leqslant \alpha_1<\ldots<\alpha_p\leqslant q
$$
we define the following ideals of the ring $A\{x_1,\ldots,x_n\}$:
the ideal $\Delta_{(\alpha)}$ is the ideal generated by the
$q$-minors of the matrix $(\partial f_{\alpha_i}/\partial
x_j)_{\alpha_i\in (\alpha),j\in [1,n]}$, the ideal $J_{\alpha}$ is
the ideal generated by $f_{\alpha_i}$ with $\alpha_i\in \alpha$.
Then we define $H_J\subseteq B$ as the image of the ideal
$\sum_{\alpha} \Delta_{(\alpha)} (J_\alpha : J)$.

\section{Scheme-Theoretic support}

\subsection{The support of singularities}

\subsubsection{Smoothness}

Let $A$ be a ring, $B$ is an $A$-algebra, and $N$ is a $B$-module.
Recall that an extension $E=(D,\varepsilon, i)$ of $B$ by $N$ is the
following exact sequence
$$
E:\; 0\to N\stackrel{i}{\longrightarrow} D
\stackrel{\varepsilon}{\longrightarrow} B \to 0
$$
where $D$ is an $A$-algebra, $\varepsilon$ is a homomorphism of
$A$-algebras, $i(N)=\ker \varepsilon$ is a square zero ideal such
that the action of $B$ on $N= N/N^2$ coincides with the initial
action on $N$. The two such extensions $E=(D,\varepsilon,i)$ and
$E'=(D',\varepsilon',i')$ are said to be isomorphic if there is an
$A$-isomorphism $\psi\colon D\to D'$ inducing the identity maps on
$N$ and $D$. Morphisms between two extensions are defined in the
obvious way.

Recall that, for any extension $E=(D,\varepsilon, i)$, any
$A$-algebras homomorphism $\phi\colon B'\to B$, and any $B$-module
homomorphism $\xi\colon N\to M$, there are unique extensions $E\phi$
and $\xi E$ such that the following diagrams are commutative
$$
\xymatrix{
    {E\phi:}&{0}\ar[r]&{N}\ar[r]\ar@{=}[d]&{D'}\ar[r]\ar[d]&{B'}\ar[r]\ar[d]&{0}\\
    {E:}&{0}\ar[r]&{N}\ar[r]^{i}&{D}\ar[r]^{\varepsilon}&{B}\ar[r]&{0}\\
}
$$
and
$$
\xymatrix{
    {E:}&{0}\ar[r]&{N}\ar[r]\ar[d]^{\xi}&{D}\ar[r]\ar[d]&{B}\ar[r]\ar@{=}[d]&{0}\\
    {\xi E:}&{0}\ar[r]&{M}\ar[r]&{D''}\ar[r]&{B}\ar[r]&{0}\\
}
$$

The set of all extensions of $B$ by $N$ will be denoted by
$\exa_A(B,N)$. It is an abelian group with respect to the Baer sum
and the multiplication by the elements of $B$ gives a structure of a
$B$-module. Now, we can define the following ideal
$$
H_{B/A} = \bigcap_{N\in
B{-}\mathrm{mod}}\annihil_B\left(\exa_A(B,N)\right)
$$

Note that, if we say that $B$ is smooth $A$-algebra, we do not
necessarily assume that $B$ is finitely generated. We do not need
this finiteness condition to describe the properties of the ideal
$H_{B/A}$. However, in all applications this condition will be
assumed and explicitly stated.

In~\cite[Section~0.2]{Elkik73}, another ideal is used instead of
$H_{B/A}$. The ideal $H_{B/A}$ was introduced in~\cite[Chapter~5,
Section~5.4]{GabRam} in a slightly different way. Due to Gabber's
ideal, we can significantly simplify computational part of methods
used in~\cite{Elkik73}.

\begin{lemma}\label{lemma:E_surj}
Let $A$ be a ring, $R$ be a smooth $A$-algebra, $J$ an ideal of $R$,
and $B$ is an $A$-algebra such that the following sequence is exact
$$
0\to J\to R\to B\to 0
$$
Then the natural map
$$
\oplus_\xi \exa_A(B,\xi)\colon \bigoplus_{\xi\in \hom_B(J/J^2,N)}
\exa_A(B,J/J^2)\to \exa_A(B,N)
$$
is surjective.
\end{lemma}
\begin{proof}
Consider the following commutative diagram
$$
\xymatrix{
    {}&{0}\ar[r]&{J}\ar[r]\ar@{-->}[d]^{\xi}&{R}\ar[r]\ar@{-->}[d]^{\phi}&{B}\ar@{=}[d]\ar[r]&{0}\\
    {E:}&{0}\ar[r]&{N}\ar[r]&{D}\ar[r]&{B}\ar[r]&{0}\\
}
$$
where the exact sequence $E$ is an element of $\exa_A(B,N)$ and we
are going to produce the dotted lines. Since $R$ is $A$-smooth,
there is an $A$-lifting $\phi\colon R\to D$ of the identity map
$Id_B$. Then $\xi$ is the restriction of $\phi$ on $J$. Since
$N^2=0$, $\phi(J^2)=0$. Therefore, we have the following commutative
diagram
$$
\xymatrix{
    {T:}&{0}\ar[r]&{J/J^2}\ar[r]\ar@{-->}[d]^{\bar\xi}&{R/J^2}\ar[r]\ar@{-->}[d]^{\bar\phi}&{B}\ar@{=}[d]\ar[r]&{0}\\
    {E=\bar \xi T:}&{0}\ar[r]&{N}\ar[r]&{D}\ar[r]&{B}\ar[r]&{0}\\
}
$$
where $\bar \phi$ and $\bar \xi$ are the maps induced by $\phi$ and
$\xi$, respectively. Hence, the element $E$ belongs to the image of
$\exa_A(B,\bar \xi)$.
\end{proof}

\begin{corollary}\label{cor:smooth_cover}
Let $A$ be a ring, $R$ a smooth $A$-algebra, $J$ an ideal of $R$,
and $B$ is an $A$-algebra such that the following sequence is exact
$$
0\to J\to R\to B\to 0
$$
Then
$$
H_{B/A} = \annihil_B\left(\exa_A(B,J/J^2)\right).
$$
\end{corollary}

\begin{lemma}\label{lemma:HB_computation}
Let $A$ be a ring, $R$ a smooth $A$-algebra, $J$ an ideal of $R$,
and $B$ is an $A$-algebra such that the following sequence is exact
$$
0\to J\to R\to B\to 0
$$
Then we have
$$
H_{B/A} = \annihil_B( \hom_{B}(J/J^2,J/J^2)/\der_{A}(R,J/J^2))
$$
\end{lemma}
\begin{proof}
By Corollary~\ref{cor:smooth_cover}, we have
$$
H_{B/A} = \annihil_B\left(\exa_A(B,J/J^2)\right).
$$
Let $x\in B$ be an arbitrary element, we can write the following
commutative diagram
$$
\xymatrix{
    {T:}&{0}\ar[r]&{J/J^2}\ar[r]\ar[d]^{x}&{R/J^2}\ar[r]\ar[d]&{B}\ar[r]\ar@{=}[d]&{0}\\
    {xT:}&{0}\ar[r]&{J/J^2}\ar[r]&{D}\ar[r]&{B}\ar[r]&{0}\\
}
$$
The element $x$ belongs $H_{B/A}$ if an only if $xT=0$. The latter
means $xT$ splits. In other words, the multiplication by $x$ extends
to an $A$-derivation
$$
d\colon R/J^2\to J/J^2.
$$
Since $J/J^2$ is a module over $R/J=B$,
$$
\der_A(R/J^2,J/J^2)=\der_A(R,J/J^2).
$$

\end{proof}

The ideal $H_{B/A}$ describes $A$-smoothness in the following
sense.
\begin{lemma}\label{lemma:HB_smooth}
Let $A$ be a ring, and $B$ a finitely generated $A$-algebra. Then,
for every prime ideal $\mathfrak p\subseteq B$, $B_\mathfrak p$ is
$A$-smooth if and only if $H_{B/A}\nsubseteq \mathfrak p$.
\end{lemma}
\begin{proof}
See Lemma~5.4.2~(ii) of~\cite[Chapter~5, Section~5.4]{GabRam}.
\end{proof}

\subsubsection{Separated differentials}

Let $A$ be a ring and $\adic\subseteq A$ is an arbitrary ideal such
that $A$ is $\adic$-adically complete. Let $B$ is a formally
finitely generated $A$-algebra. We denote by $\Omega^s_{B/A}$ the
Hausdorff quotient of $\Omega_{B/A}$. Then $\Omega^s_{B/A}$ becomes
a finite $B$-module. In particular,
$\Omega^s_{B/A}=\widehat{\Omega}_{B/A}$. If $B=A\{x\}$ is a ring of
convergent series, the module $\Omega^s_{A\{x\}/A}$ is a free module
with generators $dx$. In our particular case, (20.7.17) and
(20.7.20) of~\cite{EGA1964_20} gives the following results.

\begin{proposition}[The first fundamental sequence]\label{prop:first_fund_s}
Let $A$ be a topological ring,
$B\to C$ is a homomorphism of formally finitely generated
$A$-algebras, then
\begin{enumerate}
\item the following sequence is exact
$$
\Omega^s_{B/A}\otimes_B C \to \Omega^s_{C/A}\to \Omega^s_{C/B}\to 0
$$
\item the sequence is split exact if and only if, for every finite
$B$-module $N$, any derivation from $B$ to $N$ over $A$ extends to a
derivation from $C$ to $N$.
\end{enumerate}
\end{proposition}

\begin{proposition}[The second fundamental sequence]\label{prop:sec_fund_s}
Let $A$ be a topological ring,
$B\to C$ is a topologically surjective homomorphism of topological
$A$-algebras with the kernel $J$. Then
\begin{enumerate}
\item the following sequence is exact
$$
J/J^2\to \Omega^s_{B/A}\otimes_B C \to \Omega^s_{C/A}\to 0
$$
\item the sequence is split exact if and only if the following
sequence splits
$$
0\to J/J^2\to B/J^2\to C\to 0
$$
\end{enumerate}
\end{proposition}

\subsubsection{Formal smoothness}

Let $A$ be a ring, $\adic\subseteq A$ is an ideal such that $A$ is
$\adic$-adically complete. Let $B$ be a formally finitely generated
$A$-algebra, then we define
$$
\bar H_{B/A} = \bigcap_{M\in
B{-}\mathrm{mod}^0}\annihil_B(\exa_A(B,M))
$$
We will say that algebra $B$ is formally smooth over the complement
of $V(\adic)$ if $\adic\subseteq r(H_{B/A})$.

\begin{lemma}
Let $A$ be a ring, $\adic\subseteq A$ is an ideal such that $A$ is
$\adic$-adically complete, $B$ is formally finitely generated
$A$-algebra, and $N$ is a finite $B$-module. Assume that $R$ is the
completion of a smooth finitely generated $A$-algebra $R_0$ such
that
$$
0\to J\to R\to B\to 0
$$
Then the natural map
$$
\bigoplus_{\xi\in \hom_B(J/J^2,N)} \exa_A(B,J/J^2)\stackrel{\oplus
\exa_A(B,\xi)}{\longrightarrow} \exa_A(B,N)
$$
is surjective.
\end{lemma}
\begin{proof}
Let $N$ be an arbitrary finitely generated $B$-module and
$$
0\to N\to D\to B\to 0
$$
is an element of $\exa_A(B,N)$. Since $B$ and $N$ are Noetherian,
the ring $D$ is also Noetherian. Moreover, $B$ is $\adic$-adically
complete by definition and $N$ is finitely generated $B$-module,
thus, $N$ is also $\adic$-adically complete. Since $D$ is Noetherian
$\adic$-adic topology from $D$ induces $\adic$-adic topology on $N$.
Since completion is an exact functor, we see that $D$ is complete.

Now, consider the following diagram
$$
\xymatrix{
    {R_0}\ar[r]\ar[rd]\ar@{-->}[d]^{\phi}&{B}\ar@{=}[d]&{}\\
    {D}\ar[r]&{B}\ar[r]&{0}\\
}
$$
where $\phi$ is a lifting of the homomorphism from $R_0$ to $B$. It
exists because of $A$-smoothness of $R_0$. But $D$ is complete,
therefore, there is a unique extension $\widehat{\phi}$ of $\phi$ to
$R$. Moreover, since $\widehat{\phi}$ is a lifting, it maps $J^2$ to
zero. Now, we have the following diagram
$$
\xymatrix{
    {0}\ar[r]&{J/J^2}\ar[r]\ar[d]^{\xi}&{R/J^2}\ar[r]\ar[d]^{\varphi}&{B}\ar[r]\ar@{=}[d]&{0}\\
    {0}\ar[r]&{N}\ar[r]&{D}\ar[r]&{B}\ar[r]&{0}\\
}
$$
where $\varphi$ is the homomorphism induced by $\widehat{\phi}$, and
$\xi$ is the restriction of $\varphi$ to $J/J^2$. The rest of the
proof is a word by word repetition of the end of the proof of
Lemma~\ref{lemma:E_surj}.
\end{proof}

\begin{corollary}\label{cor:fsmooth_cover}
Let $A$ be a ring, $\adic\subseteq A$ is an ideal such that $A$ is
$\adic$-adically complete, and $B$ is a formally finitely generated
$A$-algebra. Assume that $R$ is the completion of a smooth finitely
generated $A$-algebra $R_0$ such that
$$
0\to J\to R\to B\to 0
$$
Then
$$
\bar H_{B/A} = \annihil_B(\exa_A(B,J/J^2)).
$$
\end{corollary}

\begin{lemma}\label{lemma:fHB_computation}
Let $A$ be a ring, $\adic\subseteq A$ is an ideal such that $A$ is
$\adic$-adically complete, and $B$ is a formally finitely generated
$A$-algebra. Assume that $R$ is the completion of a smooth finitely
generated $A$-algebra $R_0$ such that
$$
0\to J\to R\to B\to 0
$$
Then
$$
\bar H_{B/A} = \annihil_{B} \left(
\hom_{B}(J/J^2,J/J^2)/\der_{A}(B,J/J^2)\right)
$$
\end{lemma}
\begin{proof}
By Corollary~\ref{cor:fsmooth_cover}, we have
$$
\bar H_{B/A} = \annihil_B\left(\exa_A(B,J/J^2)\right).
$$
Let $x\in B$ be an arbitrary element, we can write the following
commutative diagram
$$
\xymatrix{
    {T:}&{0}\ar[r]&{J/J^2}\ar[r]\ar[d]^{x}&{R/J^2}\ar[r]\ar[d]&{B}\ar[r]\ar@{=}[d]&{0}\\
    {xT:}&{0}\ar[r]&{J/J^2}\ar[r]&{D}\ar[r]&{B}\ar[r]&{0}\\
}
$$
The element $x$ belongs $\bar H_{B/A}$ if an only if $xT=0$. The
latter means $xT$ splits. In other words, the multiplication by $x$
extends to an $A$-derivation
$$
d\colon R/J^2\to J/J^2.
$$
Since $J/J^2$ is a module over $R/J=B$,
$$
\der_A(R/J^2,J/J^2)=\der_A(R,J/J^2).
$$
\end{proof}

\begin{lemma}\label{lemma:H_Bness}
Let $A$ be a ring, $\adic\subseteq A$ is an ideal, $B$ is a finitely
generated $A$-algebra. Then
$$
H_{B/A}\widehat{B} = H_{B/A}\otimes_B \widehat{B} = \bar
H_{\widehat{B}/\widehat{A}}
$$
\end{lemma}
\begin{proof}
Consider an arbitrary exact sequence
$$
0\to J\to A[x]\to B\to 0
$$
where $x=\{x_1,\ldots,x_n\}$ is a finite set of indeterminates.
Then, by Lemma~\ref{lemma:HB_computation},
$$
H_{B/A} = \annihil_B
\left(\hom_B(J/J^2,J/J^2)/\hom_{B}(\Omega_{A[x]/A}\otimes_{A[x]}B,J/J^2)\right)
$$
Since $J/J^2$, $\Omega_{A[x]/A}\otimes_{A[x]}B$ are finitely
generated $B$-modules and $\widehat{B}$ is a flat $B$-module, then
\begin{multline*}
H_{B/A}\widehat{B} = H_{B/A}\otimes_{B}\widehat{B} =\\
=\annihil_{\widehat{B}}
\left(\hom_{\widehat{B}}(\widehat{J}/\widehat{J}^2,\widehat{J}/\widehat{J}^2)/\hom_{\widehat{B}}(\Omega^s_{\widehat{A}\{x\}/\widehat{A}}\otimes_{\widehat{A}\{x\}}\widehat{B},\widehat{J}/\widehat{J}^2)\right)=\\
=\annihil_{\widehat{B}}
\left(\hom_{\widehat{B}}(\widehat{J}/\widehat{J}^2,\widehat{J}/\widehat{J}^2)/\der_{\widehat{A}}(\widehat{B},\widehat{J}/\widehat{J}^2)\right)=\\
= \bar H_{\widehat{B}/\widehat{A}}
\end{multline*}
The latter equality holds because of
Lemma~\ref{lemma:fHB_computation}.
\end{proof}

\begin{remark}
Lemma~\ref{lemma:H_Bness} shows that if an $\widehat{A}$-algebra
$\bar B$ is a completion of a finitely generated $A$-algebra being
smooth over the complement of $V(\adic)$, then $\bar B$ is formally
smooth over the complement of $V(\widehat{\adic})$. So, the
condition of formal smoothness is necessary for algebraization of an
$\widehat{A}$-algebra $\bar B$ by a smooth $A$-algebra.
\end{remark}

\begin{lemma}\label{lemma:fsmooth_Elkik}
Let $A$ be a ring $\adic\subseteq A$ is an ideal such that $A$ is
$\adic$-adically complete, and $B$ is a formally finitely generated
$A$-algebra presented in the following form
$$
B = A\{x_1,\ldots, x_n\}/J,\quad \mbox{ where } \quad J =
(f_1,\ldots, f_q)
$$
Then
$$
H_J \subseteq \bar H_{B/A}\quad \mbox{ and } \quad r(H_J) = r(\bar
H_{B/A})
$$
\end{lemma}
\begin{proof}
The proof of the inclusion $H_J\subseteq \bar H_{B/A}$ is a word by
word repetition of the proof of Lemma~5.4.6 of~\cite{GabRam}, where
we should use $\Omega^s$ instead of $\Omega$.

To proof coincidence of the radicals, it is enough to show that, if
a prime ideal $\mathfrak p\subseteq B$ does not contain $\bar
H_{B/A}$, then it does not contain $H_J$. Let $t\in \bar
H_{B/A}\setminus \mathfrak p$. Then the map $J/J^2\to J/J^2$ given
by $x\mapsto tx$ lifts to a derivation $A\{x\}\to J/J^2$ be
definition. In particular, the sequence
$$
0\to J_{\mathfrak p}/J^2_{\mathfrak p} \to
\Omega^s_{A\{x\}/A}\otimes_{A\{x\}} B_{\mathfrak p}\to
(\Omega^s_{B/A})_{\mathfrak p} \to 0
$$
is split exact. Since $B_p$ is local and $f_1,\ldots,f_q$ generate
$J$, after reordering $f_i$, we may assume that, for some $k$,
images of $f_1,\ldots,f_k$ form a basis of $J_{\mathfrak
p}/J^2_{\mathfrak p}$. Since $B_\mathfrak p$ does not contain
idempotents except $0$ and $1$, elements $f_1,\ldots,f_k$ generate
$J_\mathfrak p$. In particular, there exists an element $t\in
B\setminus \mathfrak p$ such that $tJ\subseteq (f_1,\ldots,f_k)$.

Since $B_{\mathfrak p}$ is local, after reordering of $x_i$, we may
assume that
$$
df_1,\ldots,df_k,dx_{k+1},\ldots,dx_n
$$
is a basis of
$$
\Omega^s_{A\{x\}/A}\otimes_{A\{x\}} B_{\mathfrak p}
$$
because $dx_i$ are generators, and the set $df_1,\ldots, df_k$ is a
basis of a direct summand. Hence the element $h=\det (\partial
f_i/\partial x_j)_{i,j\leqslant k}$ is invertible in $B_{\mathfrak
p}$, thus, is not in $\mathfrak p$. Therefore, the element $ht$ is
not in $\mathfrak p$ and belongs to $H_J$.

\end{proof}

\begin{remark}
In particular, the notion of formal smoothness used
in~\cite[Section~III.4]{Elkik73} coincides with the one used in this
paper. Hence we can formally cite results from~\cite{Elkik73}.
\end{remark}

\subsection{The support of non-projectivity}

In this section, we define an analogue of the Gabber's ideal for
modules. This ideal is used instead of the Fitting ideal to
algebraize finite modules. It describes points where our module is
not projective.

Let $A$ be a ring and $P$ is an arbitrary $A$-module. Then we define
the ideal $H_P$ by the formula
$$
H_P = \bigcap_{M\in A-\mathrm{mod}}\annihil_A
\left(\ext^1_A(P,M)\right)
$$

\begin{lemma}
Let $A$ be a ring and we are given an exact sequence of $A$-modules
$$
0\to K\to F_0\to P\to 0
$$
where $F_0$ is projective. Then the map
$$
\oplus_g\ext^1_A(P,g)\colon\bigoplus_{g\in \hom_A (K, M)}
\ext^1_A(P,K)\to \ext^1_A(P,M)
$$
is surjective.
\end{lemma}
\begin{proof}
Consider the following projective resolution for $P$
$$
\xymatrix{
    {F_2}\ar[r]^{d_2}&{F_1}\ar[d]^{\xi}\ar[r]&{F_0}\ar@{=}[d]\ar[r]&{P}\ar[r]\ar@{=}[d]&{0}\\
    {0}\ar[r]&{K}\ar[r]&{F_0}\ar[r]&{P}\ar[r]&{0}\\
}
$$
where $K$ is the cokernel of $d_2$ and $\xi$ is the corresponding
quotient homomorphism. The map $\xi$ corresponds to an element
$\bar\xi$ in $\ext^1_A(P,K)$. Let $\varphi\colon F_1\to M$ be a
homomorphism corresponding to some element $\bar\varphi$ of
$\ext^1_A(P,M)$, then, by definition of $\xi$, there is a
homomorphism $g\colon K\to M$ such that $g\circ\xi = \varphi$. Thus,
the element $\bar \varphi$ belongs to the image of $\ext^1_A(P,g)$.
\end{proof}

\begin{corollary}\label{cor:H_P_K}
Let $A$ be a ring and we are given an exact sequence of $A$-modules
$$
0\to K\to F_0\to P\to 0,
$$
where $F_0$ is projective. Then
$$
H_P = \annihil_A \left(\ext^1_A(P,K)\right).
$$
\end{corollary}

\begin{corollary}
Let $A$ be a ring and $P$ a finite $A$-module, then
$$
H_P = \bigcap_{M\in A{-}\mathrm{mod}^0}\annihil_A\left(
\ext^1_A(P,M)\right).
$$
\end{corollary}

The following proposition explains the geometric meaning of the
ideal $H_P$.

\begin{proposition}\label{lemma:H_P_proj}
Let $A$ be a ring, $P$ a finite $A$-module, and $\mathfrak
p\subseteq A$ is a prime ideal. Then $H_P\nsubseteq \mathfrak p$ if
and only if $P_\mathfrak p$ is projective $A_\mathfrak p$-module.
\end{proposition}
\begin{proof}

Suppose that $H_P\nsubseteq \mathfrak p$ and let $t\in
H_P\setminus\mathfrak p$. Consider some exact sequence $0\to K\to
F\to P\to 0$, where $F$ is a finite free $A$-module. By the
definition of $H_P$, the homomorphism $t\colon K\to K$ lifts to a
homomorphism $F\to K$. After the localization by $\mathfrak p$, the
multiplication by $t$ becomes an isomorphism. So, the sequence $0\to
K_\mathfrak p\to F_\mathfrak p\to P_\mathfrak p\to 0$ splits.

Conversely, suppose $P_\mathfrak p$ is projective. By
Corollary~\ref{cor:H_P_K}, it is enough to show that there is an
$s\notin \mathfrak p$ such that $s\colon K\to K$ lifts to a
homomorphism $F\to K$. Since the sequence $0\to K_{\mathfrak p}\to
F_{\mathfrak p}\to P_{\mathfrak p}\to 0$ splits, there is a
splitting map $\psi\colon F_\mathfrak p\to K_\mathfrak p$. The
module $F$ is finite, hence, there is an element $s\notin \mathfrak
p$ such that $s\psi\colon F\to K$ is a well-defined lifting of
$s\colon K\to K$.

\end{proof}

\begin{proposition}\label{prop:H_Omega}
Let $A$ be a ring, $\adic\subseteq A$ is an ideal such that $A$ is a
$\adic$-adically complete ring, and $B$ is a formally finitely
generated $A$-algebra. Then $\bar H_{B/A}\subseteq
H_{\Omega^s_{B/A}}$.
\end{proposition}
\begin{proof}
We represent our algebra $B$ as the following quotient
$$
0\to J\to A\{x\}\to B\to 0
$$
where $x=\{x_1,\ldots,x_n\}$ is a finite set of indeterminates. By
the second fundamental sequence, we have
$$
\xymatrix{
    {0}\ar[r]&{K}\ar[r]&{J/J^2}\ar[r]&{\Omega^s_{A\{x\}/A}\otimes_{A\{x\}} B}\ar[r]&{\Omega^s_{B/A}}\ar[r]&{0}\\
}
$$
Then by Lemmas~\ref{lemma:HB_computation}
and~\ref{lemma:fHB_computation}, we have
\begin{align*}
\exa_A\left(B,{J/J^2}\right) &= \hom_{B}\left(J/J^2,{J/J^2}\right)/\hom_{B}\left(\Omega^s_{A\{x\}/A}\otimes_{A\{X\}} B, {J/J^2}\right)\\
\ext^1_B\left(\Omega^s_{B/A},{J/J^2}\right) &=
\hom_{B}\left(\left(J/J^2\right)/K,{J/J^2}\right)/\hom_{B}\left(\Omega^s_{A\{x\}/A}\otimes_{A\{x\}}
B, {J/J^2}\right)
\end{align*}
Hence, the inclusion
$$
\ext^1_B\left(\Omega_{B/A},{J/J^2}\right)\subseteq
\exa_A\left(B,{J/J^2}\right)
$$
holds. Therefore, we have
$$
\bar H_{B/A}
=\annihil_B\left(\exa_A\left(B,{J/J^2}\right)\right)\subseteq
\annihil_B \left(\ext^1_B\left(\Omega^s_{B/A},{J/J^2}\right)\right)
= H_{\Omega^s_{B/A}}
$$
\end{proof}

\begin{remark}
Similarly to the proof of Proposition~\ref{prop:H_Omega}, one can
show the following statement. If $A$ is a ring and $B$ is a finitely
generated $A$-algebra, then
$$
H_{B/A}\subseteq H_{\Omega_{B/A}}
$$
But we will not need this fact.
\end{remark}

\section{Lifting of pairs of homomorphisms}

\subsection{Lifting of module homomorphisms}

\begin{lemma}\label{lemma:induced_struct}
Let $A$ be a ring, $\adic=(t)\subseteq A$ a principal ideal, and we
are given an exact sequence of $A$-modules
$$
0\to K\to M\stackrel{\varphi}{\longrightarrow}N\to 0
$$
Let $\varphi_n$ be the restriction of $\varphi$ on $\adic ^n M$ and
$\varphi^d$ is the induced map from $M/\adic^dM\to N/\adic^d N$.
Suppose that, for some natural number $c$ and for all $n\geqslant
c$, we have the following family of exact sequences
$$
0\to \adic ^{n-c}P_n\to \adic ^n
M\stackrel{\varphi_n}{\longrightarrow}\adic ^n N\to 0,
$$
where $P_n\subseteq \ker\varphi$. Then, for every pair $(n,d)$ with
$n\leqslant d$, we also have the following family of exact sequences
$$
0\to \adic ^{n-c}P_n^d \to \adic ^nM/\adic ^d
M\stackrel{\varphi_{n}^d}{\longrightarrow}\adic ^n N/\adic ^d N\to
0,
$$
where $\varphi_{n}^d$ is the induced by $\varphi$ map and
$P_n^d\subseteq \ker \varphi_{n}^d$.

Moreover, if for any two pairs $(n,d)$ and $(n',d')$ such that
$n\geqslant n'$ and $d\geqslant d'$, $P_n\subseteq P_{n'}$, the
corresponding quotient map induces a well-defined map $P_n^d\to
P_{n'}^{d'}$.
\end{lemma}
\begin{proof}
We define $P_n^d$ as the image of $P_n$ under the corresponding
quotient map. Now, consider an element $p$ of the kernel of
${\varphi}_{n}^d$. By definition, $p=t^n a$, where $a\in M$. Then
${\varphi}^d(t^n a)=0$ in $N/\adic ^dN$. Hence
$$
\varphi(t^n a) = t^d h = t^d\varphi(b).
$$
Therefore, $\varphi(t^n a - t^d b) = 0$. By the hypothesis, the
element $k = t^n (a - t^{d-n} b)$ can be presented as $t^{n-c}r$,
where $r\in P_n$. So, we have $p=t^{n-c}r$ in $M/\adic ^d M$.
\end{proof}

\begin{lemma}\label{lemma_thg_lift}
Let $A$ be a ring, $\adic=(t)\subseteq A$ a principal ideal, and we
are given the following exact sequence of $A$-modules
$$
0\to R\to M\stackrel{\varphi}{\longrightarrow} N\to 0
$$
Let $r$ and $c$ be integral numbers such that
\begin{enumerate}
\item $\adic ^rM\cap \annihil_M(t)=0$;
\item $\adic ^n M \cap R = \adic^{n-c}(\adic ^c M\cap R)$ for all $n\geqslant
c$.
\end{enumerate}
Let $k$, $n$, and $h$ be a triple of integral numbers satisfying the
conditions
$$
k>n, \;\;\; n>c+h, \;\;\; n>r+h.
$$
On the following diagram, $\varphi_{n,k}$ is the map induced by
$\varphi$ and $R_{n,k}$ is its kernel
$$
0\to R_{n,k} \to \adic^n M/\adic^k M
\stackrel{\varphi_{n,k}}{\longrightarrow}\adic^n N/\adic^k N \to 0
$$
Suppose also that the following diagram with solid arrows only is
given
$$
\xymatrix{
    {}&{R_{n,k}}\ar[r]^{\nu}&{R_{n-h,k}}\ar[r]^{\mu}&{R_{n-h,k-h}}\\
    {A^d}\ar[r]^{Q}&{A^p}\ar[u]^{g}\ar@{-->}[rru]_{t^h g'}&{}&{}\\
}
$$
where $\nu$ and $\mu$ are the natural maps, and we have $gQ=0$. Then
there exists a homomorphism
$$
g'\colon A^p\to R_{n-h,k-h}
$$
such that $g'Q=0$ and $\mu\circ\nu\circ g = t^h g'$.
\end{lemma}
\begin{proof}
Firstly, by Lemma~\ref{lemma:induced_struct}, $R_{n,k} = \adic
^{n-c}\widetilde{R}_{n,k}$, where $\widetilde{R}_{n,k}\subseteq \ker
\varphi_{n,k}$. Thus, we have the following diagram
$$
\xymatrix{
    {}&{\adic ^{n-c}\widetilde{R}_{n,k}}\ar[r]^{\nu}&{\adic ^{n-h-c}\widetilde{R}_{n-h,k}}\ar[r]^{\mu}&{\adic ^{n-h-c}\widetilde{R}_{n-h,k-h}}\\
    {A^d}\ar[r]^{Q}&{A^p}\ar[u]^{g}&{}&{}\\
}
$$
Hence, we have an equality $\nu \circ g = t^h \bar f$ for some $\bar
f\colon A^p\to \adic ^{n-h-c}\widetilde{R}_{n-h,c}$. Its lifting
$A^p\to \adic ^{n-h}M$ will be denoted by $f$.

We will show that $\mu\circ \bar f\circ Q = 0$. Let $q_1,\ldots,
q_d$ be the generators of the image of $Q$. Then, we know that
$$
t^h \bar f(q_i)=0\; \mbox{ in } \: \adic ^{n-h}M/\adic ^{k}M.
$$
Hence, for some $m_i\in M$, we have
$$
t^h f(q_i) = t^k m_i\;\mbox{ in } \; M.
$$
Thus,
$$
t^h(f(q_i)-t^{k-h}m_i)=0\; \mbox{ in } \; M.
$$
By the definition of $f$, the elements $f(q_i)-t^{k-h}m_i$ belong to
$\adic ^{n-h}M\subseteq \adic ^rM$. By the definition of $r$, we
have
$$
f(q_i)=t^{k-h}m_i\; \mbox{ in } \; M.
$$
And thus
$$
\mu\circ\bar f(q_i)= 0 \; \mbox{ in } \; \adic ^{n-h}M/\adic
^{k-h}M.
$$
\end{proof}

\begin{corollary}\label{corollary_thg_lift}

Let $A$ be a ring, $\adic=(t)\subseteq A$ a principal ideal, and $M$
is an $A$-module. Let $r$, $n$, $k$, and $h$ be natural numbers such
that
\begin{enumerate}
\item $\adic ^rM\cap \annihil_M(t)=0$;
\item $k>n$ and $n> r+h$.
\end{enumerate}
Assume also that  the following diagram with solid arrows only is
given
$$
\xymatrix{
    {}&{\adic ^nM/\adic ^kM}\ar[r]^{\nu}&{\adic ^{n-h}M/\adic ^{k}M}\ar[r]^{\mu}&{\adic ^{n-h}M/\adic ^{k-h}M}\\
    {A^d}\ar[r]^{Q}&{A^p}\ar[u]^{g}\ar@{-->}[rru]_{t^h g'}&{}&{}\\
}
$$
where $\nu$ and $\mu$ are the natural maps, and we have $gQ=0$. Then
there exists a homomorphism
$$
g'\colon A^p\to \adic ^{n-h} M/\adic ^{k-h} M
$$
such that $g'Q=0$ and $\mu\circ\nu\circ g = t^h g'$.
\end{corollary}
\begin{proof}
We should take $N = 0$ in Lemma~\ref{lemma_thg_lift}. Just note
that, in this case, $c=0$.
\end{proof}

\begin{lemma}\label{lemma_duo_lift}

Let $A$ be a ring, $\adic=(t)\subseteq A$ a principal ideal, we are
given the following exact sequence of $A$-modules
$$
0\to R\to M\stackrel{\varphi}{\longrightarrow} N\to 0
$$
and $P$ is a finite $A$-module such that $\adic^h\subseteq H_P$. Let
$r$ and $c$ be integral numbers such that
\begin{enumerate}
\item $\adic ^rM\cap \annihil_M(t)=0$;
\item $\adic ^n M \cap R = \adic^{n-c}(\adic ^c M\cap R)$ for all $n\geqslant
c$.
\end{enumerate}
Let $k$ and $n$ be a pair of integers such that
$$
k>n, \;\;\; n>c+h, \;\;\; n>r+h.
$$
On the following diagram, $\varphi_{n,k}$ is the map induced by
$\varphi$ and $R_{n,k}$ is its kernel
$$
0\to R_{n,k} \to \adic^n M/\adic^k M
\stackrel{\varphi_{n,k}}{\longrightarrow}\adic^n N/\adic^k N \to 0
$$
Then, on the following diagram, for every homomorphism $\psi$, there
exists a homomorphism $\psi'$ such that the diagram is commutative
$$
\xymatrix@C=60pt{
    {\adic ^{n-h}M/\adic ^{k-h}M}\ar[r]^{\varphi_{n-h,k-h}}&{\adic ^{n-h}N/\adic ^{k-h}N}\\
    {P}\ar[r]^{\psi}\ar@{-->}[u]^{\psi'}&{\adic ^n N/\adic ^k N}\ar[u]\\
}
$$
\end{lemma}
\begin{proof}
We take an arbitrary finite resolution of $P$ of the form
$$
\ldots\to A^d\stackrel{Q}{\longrightarrow} A^p \to A^q \to P \to 0
$$
Then, the homomorphism $\psi$ extends to a homomorphism of the
resolution as follows
$$
\xymatrix{
    {0}\ar[r]&{R_{n-h,k-h}}\ar[r]^-{i}&{\adic ^{n-h}M/\adic ^{k-h}M}\ar[r]&{\adic ^{n-h}N/\adic ^{k-h}N}\ar[r]&{0}\\
    {0}\ar[r]&{R_{n-h,k}}\ar[r]\ar[u]^{\mu}&{\adic ^{n-h}M/\adic ^{k}M}\ar[r]\ar[u]^{\bar\mu}&{\adic ^{n-h}N/\adic ^{k}N}\ar[r]\ar[u]&{0}\\
    {0}\ar[r]&{R_{n,k}}\ar[r]\ar[u]^{\nu}&{\adic ^n M/\adic ^k M}\ar[r]\ar[u]^{\bar \nu}&{\adic ^n M/\adic ^k N}\ar[r]\ar[u]&{0}\\
    {A^d}\ar[r]^{Q}&{A^p}\ar[r]\ar@{-->}[u]^{g}\ar@/^30pt/@{..>}^{t^hg'}[uuu]^{}&{A^q}\ar[r]\ar@{-->}[u]^{f}\ar@/^6pt/@{..>}[uuul]^{\bar g}&{P}\ar[r]\ar[u]^{\psi}\ar@{..>}[uuul]^{\psi'}&{0}\\
}
$$
On this diagram, arrows $\mu$, $\nu$, $\bar \mu$, and $\bar \nu$ are
the natural maps. Now, we will produce the dotted arrows.

By Lemma~\ref{lemma_thg_lift}, there is a homomorphism $g'\colon
A^p\to R_{n-h,k-h}$ such that $g'\circ Q = 0$ and $\mu\circ\nu\circ
g = t^h g'$. Since $t^h\in H_P$, the homomorphism $\mu\circ\nu\circ
g$ lifts to a homomorphism $\bar g\colon A^p\to R_{n-h,k-h}$. Now,
the homomorphism
$$
\bar\mu\circ\bar\nu\circ f - i\circ \bar g\colon A^p \to \adic
^{n-h}M/\adic ^{k-h}M
$$
induces the required one $\psi'\colon P\to \adic ^{n-h} M/ \adic
^{k-h} M$.
\end{proof}

\begin{lemma}\label{lemma_step_lift}
Let $A$ be a ring, $\adic=(t)\subseteq A$ a principal ideal, $M$ is
an $A$-module, and $P$ is a finite $A$-module such that
$\adic^h\subseteq H_P$. Let $r$, $n$,  and $k$ be natural numbers
such that
\begin{enumerate}
\item $\adic ^rM\cap \annihil_M(t)=0$;
\item $k>n+h$ and $n> r+h$.
\end{enumerate}
Then, on the following diagram, for every homomorphism $\psi$, there
exists a homomorphism $\psi'$ such that the diagram is commutative
$$
\xymatrix{
    {M/\adic^{k-h}M}\ar[r]&{M/\adic^{n-h}M}\\
    {P}\ar[r]^{\psi}\ar@{-->}[u]^{\psi'}&{M/\adic^{n}M}\ar[u]
}
$$
\end{lemma}
\begin{proof}
We take an arbitrary free resolution of $P$ as follows
$$
\ldots\to A^d\stackrel{Q}{\longrightarrow}A^p\to A^q\to P\to 0
$$
Then the homomorphism $\psi$ lifts to a homomorphism of the
resolution as follows
$$
\xymatrix{
    {0}\ar[r]&{\adic ^{n-h}M/\adic ^{k-h}M}\ar[r]^-{i}&{M/\adic ^{k-h}M}\ar[r]&{M/\adic ^{n-h}M}\ar[r]&{0}\\
    {0}\ar[r]&{\adic ^nM/\adic ^{k}M}\ar[r]\ar[u]^{\nu}&{M/\adic ^{k}M}\ar[r]\ar[u]^{\bar\nu}&{M/\adic ^{n}M}\ar[r]\ar[u]&{0}\\
    {A^d}\ar[r]^{Q}&{A^p}\ar[r]\ar[u]^{g}\ar@/^40pt/@{..>}[uu]^(0.70){t^hg'}&{A^q}\ar[r]\ar[u]^{f}\ar@/^6pt/@{..>}[luu]^(0.40){\varphi}&{P}\ar[r]\ar[u]^{\psi}\ar@{..>}[uul]^(0.40){\psi'}&{0}\\
}
$$
On this diagram, $\nu$ and $\bar \nu$ are the natural maps. Now, We
will produce the dotted arrows.

By Corollary~\ref{corollary_thg_lift}, there is a homomorphism
$$
g'\colon A^p\to \adic ^{n-h}M/\adic ^{k-h}M
$$
such that $\nu\circ g = t^h g'$ and $g'\circ Q = 0$. Since $t^h\in
H_P$, the homomorphism $\nu\circ g$ lifts to a homomorphism
$\varphi\colon A^q\to \adic ^{n-h}M/\adic ^{k-h}M$. Then the
homomorphism
$$
\bar\nu\circ f - i \circ \varphi\colon A^p\to M/\adic ^{k-h}M
$$
induces the required lifting $\psi'\colon P\to M/\adic ^{k-h}M$.
\end{proof}

\begin{theorem}\label{theorem:duo_mod_lift}

Let $A$ be a ring, $\adic=(t)\subseteq A$ a principal ideal, $M$ is
an $\adic$-adically complete $A$-module, and $P$ is a finite
$A$-module such that $\adic^h\subseteq H_P$. Let $r$ and $n$ be
natural numbers such that
\begin{enumerate}
\item $\adic ^rM\cap \annihil_M(t)=0$;
\item  $n > \max(r+h, 2h)$.
\end{enumerate}
Then, on the following diagram, for every homomorphism $\psi$, there
exists a homomorphism $\psi'$ such that the diagram is commutative
$$
\xymatrix{
    {M}\ar[r]&{M/\adic^{n-h}M}\\
    {P}\ar[r]^{\psi}\ar@{-->}[u]^{\psi'}&{M/\adic^n M}\ar[u]
}
$$
\end{theorem}
\begin{proof}
By Lemma~\ref{lemma_step_lift}, we see that, for an arbitrary
$n>\max(r+h,2h)$ and every homomorphism $\psi$ as above, there is a
homomorphism
$$
\psi'\colon P\to M/\adic ^{2(n-h)}M
$$
such that the following diagram is commutative
$$
\xymatrix{
    {M/\adic^{2(n-h)}M}\ar[r]&{M/\adic^{n-h}M}\\
    {P}\ar[r]^-{\psi}\ar[u]^-{\psi'}&{M/\adic^n M}\ar[u]
}
$$
Then, since $M$ is complete and $2(n-h)>n$, we get the result by the
induction on $n$.

\end{proof}

\subsection{Lifting of ring homomorphisms}

\begin{theorem}\label{theorem:common_alg_lifting}
Let $A$ be a ring, $\adic=(t)\subseteq A$ a principal ideal such
that $A$ is $\adic$-adically complete. Let $\pi \colon B\to \bar B$
be a surjective $A$-homomorphism of formally finitely generated
$A$-algebras, $C$ and $\bar C$ are $A$-algebras such that
\begin{enumerate}
\item There is an exact sequence of $A$-algebras
$$
0\to I\to C\stackrel{\varphi}{\longrightarrow} {\bar C}\to 0;
$$
\item $C$ is $\adic C$-adically complete;
\end{enumerate}
Suppose that we are given natural numbers $r$, $h$, $c$, and $n$
such that
\begin{enumerate}
\item $\adic ^r C\cap \annihil_C(t)=0$ and $\adic ^{ r}{\bar C}\cap \annihil_{\bar C}(t)=0$;
\item $\adic ^k C\cap I = \adic ^{k-c}(\adic ^c C\cap I)$ for all $k\geqslant c$;
\item $\adic ^{h}\subseteq \bar H_{B/A}\subseteq B$ and $\adic ^{ h}\subseteq \bar H_{\bar B/A}\subseteq \bar B$;
\item $n>\max(r+2h,c+2h,4h)$.
\end{enumerate}
Let we be given the following commutative diagram
$$
\xymatrix{
    {C/\adic ^n C}\ar[r]^{\varphi_n}&{{\bar C}/\adic ^n {\bar C}}\ar[r]&{0}\\
    {B}\ar[r]^{\pi}\ar[u]^{\alpha_n}&{\bar B}\ar[u]^{\beta_n}&{}\\
}
$$
where $\varphi_n$ is induced by $\varphi$.

Then, there exist $A$-homomorphisms $\alpha\colon B\to C$ and
$\beta\colon \bar B\to {\bar C}$ such that the following diagram is
commutative
$$
\xymatrix{
    {B}\ar[dd]_{\alpha_n}\ar[rd]^{\pi}\ar@{-->}[rr]^{\alpha}&{}&{C}\ar[rd]^{\varphi}\ar[dd]|!{[d];[dr]}\hole&{}\\
    {}&{\bar B}\ar[dd]_(.30){\beta_n}\ar@{-->}[rr]^(.30){\beta}&{}&{{\bar C}}\ar[dd]\\
    {C/\adic ^n C}\ar[dr]^{\varphi_n}\ar[rr]|!{[r];[rd]}\hole&{}&{C/\adic ^{n-2h}C}\ar[rd]^{\varphi_{n-2h}}&{}\\
    {}&{{\bar C}/\adic ^n {\bar C}}\ar[rr]&{}&{{\bar C}/\adic ^{n-2h}{\bar C}}\\
}
$$
\end{theorem}
\begin{proof}
By Lemma~1 of~\cite{Elkik73}, one finds $A$-homomorphisms
$$
\gamma\colon B\to C/\adic ^{2(n-h)}C \;\;\mbox{ and }
\;\;\delta\colon \bar B\to C/\adic ^{2(n-h)}{\bar C}
$$
such that the following diagrams are commutative
$$
\xymatrix{
    {B}\ar@{-->}[r]^-{\gamma}\ar[d]_{\alpha_n}\ar[rd]^{\alpha_{n-h}}&{C/\adic ^{2(n-h)}C}\ar[d]\\
    {C/\adic ^n C}\ar[r]&{C/\adic ^{n-h}C}\\
}\:\:\: \xymatrix{
    {\bar B}\ar@{-->}[r]^-{\delta}\ar[d]_{\beta_n}\ar[rd]^{\beta_{n-h}}&{{\bar C}/\adic ^{2(n-h)}{\bar C}}\ar[d]\\
    {{\bar C}/\adic ^n {\bar C}}\ar[r]&{{\bar C}/\adic ^{n-h}{\bar C}}\\
}
$$
Where, $\alpha_{n-h}$ and $\beta_{n-h}$ are the compositions of
$\alpha_n$ and $\beta_n$ with the corresponding quotient map,
respectively.

We will draw everything on the following diagram
$$
\xymatrix@!0@R=40pt@C=50pt{
    {\adic ^{n-2h}{C}/\adic ^{2(n-2h)}{C}}\ar[rrr]\ar[rrd]^{i}&&&{\adic ^{n-2h}{\bar C}/\adic ^{2(n-2h)}{\bar C}}\ar[rrd]&&&&\\
    &&{{C}/\adic ^{2(n-2h)}{C}}\ar[rrr]^{\varphi_{2(n-2h)}}\ar[rrd]&&&{{\bar C}/\adic ^{2(n-2h)}{\bar C}}\ar[rrd]&&\\
    &&&&{{C}/\adic ^{n-2h}{C}}\ar[rrr]^{\varphi_{n-2h}}&&&{{\bar C}/\adic ^{n-2h}{\bar C}}\\
    {\adic ^{n-h} {C}/\adic ^{2(n-h)}{C}}\ar[rrr]|!{[rr];[rrd]}\hole\ar[uuu]\ar[rrd]&&&{\adic ^{n-h}{\bar C}/\adic ^{2(n-h)}{\bar C}}\ar[rrd]_(0.30){j}|!{[r];[rd]}\hole\ar[uuu]|!{[uul];[uuulll]}\hole|!{[uu];[uur]}\hole&&&&\\
    &&{{C}/\adic ^{2(n-h)}{C}}\ar[rrr]^{\varphi_{2(n-h)}}|!{[rr];[rrd]}\hole\ar[uuu]^{\nu}\ar[rrd]&&&{{\bar C}/\adic ^{2(n-h)}{\bar C}}\ar[rrd]\ar[uuu]^{\bar \nu}|!{[uu];[uur]}\hole&&\\
    &{\Omega^s_{B/A}}\ar@/^30pt/@{..>}[uurr]^{\phi_d}\ar@/_15pt/@{..>}[uuuuul]_{\phi_{d'}}&&&{{C}/\adic ^{n-h} {C}}\ar[rrr]^{\varphi_{n-h}}|!{[r];[rd]}\hole\ar[uuu]&&&{{\bar C}/\adic ^{n-h} {\bar C}}\ar[uuu]\\
    &&{B}\ar[lu]^{d_{B/A}}\ar[rrr]^{\pi}\ar@{-->}[uu]^(.65){\gamma}\ar[rru]^{\alpha_{n-h}}\ar@{..>}[ruuu]_<<<<<<<<<<{d}&&&{\bar B}\ar@{-->}[uu]^(.65){\delta}\ar[rru]^{\beta_{n-h}}&&\\
}
$$
On this diagram $i$, $j$, $\nu$, $\bar \nu$ are the natural maps. We
will produce all dotted arrows step by step. Before this, let us
note that the diagram is commutative only for the solid arrows. The
four lines of this diagram (going from the back facet to the front
facet) are the extensions of rings by modules. Therefore, all
modules on the back facet are $B$-modules.

By construction, the difference of homomorphisms
$$
\varphi_{2(n-h)}\circ\gamma - \delta\circ\pi
$$
factors through $j$, that is, equals $j\circ d$ for some
$$
d\colon B\to \adic ^{n-h}{\bar C}/\adic ^{2(n-h)}{\bar C}
$$
and $d$ is a derivation of $B$ over $A$. The derivation $d$
corresponds to a $B$-homomorphism $\phi_d$. Then, by
Lemma~\ref{lemma_duo_lift}, we find a $B$-homomorphism
$$
\phi_{d'}\colon \Omega^s_{B/A}\to \adic ^{n-2h}C/\adic ^{2(n-2h)}C.
$$
This homomorphism corresponds to the derivation $d'=\phi_{d'}\circ
d_{B/A}$, where
$$
d'\colon B\to \adic ^{n-2h}C/\adic ^{2(n-2h)}C.
$$
Thus, the map
$$
\gamma'=\nu\circ\gamma-i\circ d'\colon B\to C/\adic ^{2(n-2h) }C
$$
is a homomorphism  and $\gamma'$ together with
$$
\bar\nu\circ\delta\colon \bar B\to {\bar C}/\adic ^{2(n-2h)}{\bar C}
$$
gives the required lifting. Now, since $C$ and, hence, ${\bar C}$
are $\adic $-adically complete, the result follows by induction on
$n$.

\end{proof}

\begin{remark}
In the proof of Theorem~\ref{theorem:common_alg_lifting}, we cannot
cite Lemma~1 of~\cite{Elkik73} because they use the ideal $H_J$
instead of $\bar H_{B/A}$. There are two possible solutions of this
issue. The first one is to see that the proof of Lemma~1 is word by
word applicable for the ideal $\bar H_{B/A}$. The second way is to
use Lemma~\ref{lemma:fsmooth_Elkik} saying that, for some natural
$l$, we have $\bar H_{B/A}^l\subseteq H_J$. Then
$\adic^{hl}\subseteq H_J$. However, in this case, we should use $hl$
instead of $h$.
\end{remark}

\begin{theorem}\label{theorem:com_alg_lift_h}
Let $A$ be a ring $\adic=(t)\subseteq A$ a principal ideal. Let
$\phi\colon B\to \bar B$ be a surjective $A$-homomorphism of
finitely generated $A$-algebras, $C$ and $\bar C$ are $A$-algebras
such that
\begin{enumerate}
\item
$B$ and $\bar B$ are smooth over the complement of $V(\adic)$;

\item
There is an exact sequence of $A$-algebras
$$
0\to I\to C\stackrel{\varphi}{\longrightarrow} \bar C\to 0;
$$

\item
The pairs $(C,\adic C)$ and $(\bar C,\adic \bar C)$ are henselian.
\end{enumerate}
Suppose we are given homomorphisms $\varepsilon \colon
\widehat{B}\to \widehat{C}$ and $\bar \varepsilon \colon
\widehat{\bar B}\to \widehat{\bar C}$ such that the following
diagram with the solid arrows only is commutative
$$
\xymatrix@R=10pt@C=10pt{
    &&&&{\widehat{C}}\ar[dll]&&&&&{\widehat{B}}\ar[dll]^{\widehat{\phi}}\ar[lllll]_{\varepsilon}\\
    &&{\widehat{\bar C}}\ar[dll]&&&&&{\widehat{\bar B}}\ar[dll]\ar[lllll]_{\bar\varepsilon}&&\\
    0&&&&&0&&&&\\
    &&&&{C}\ar[dll]\ar[uuu] |!{[uu];[uur]}\hole&&&&&{B}\ar[dll]^{\phi}\ar[uuu]\ar@{-->}[lllll]_{\gamma} |!{[ll];[llu]}\hole\\
    &&{\bar C}\ar[dll]\ar[uuu]&&&&&{\bar B}\ar[dll]\ar[uuu]\ar@{-->}[lllll]_{\bar\gamma}&&\\
    0&&&&&0&&&&\\
}
$$
Then, for every natural number $n$, there exist two homomorphisms
$\gamma\colon B\to C$ and $\bar \gamma \colon \bar B\to \bar C$ such
that the bottom square on the previous diagram is commutative and
the following diagrams are commutative
$$
\xymatrix@!0@C=40pt{
    {}&{\widehat{C}}\ar[rd]&{}\\
    {\widehat{B}}\ar[ru]^{\varepsilon}\ar[rd]_{\widehat{\gamma}}&{}&{\widehat{C}/\adic^{n}\widehat{C}}\\
    {}&{\widehat{C}}\ar[ru]&{}\\
}
 \quad
\xymatrix@!0@C=40pt{
    {}&{\widehat{\bar C}}\ar[rd]&{}\\
    {\widehat{\bar B}}\ar[ru]^{\bar \varepsilon}\ar[rd]_{\widehat{\bar\gamma}}&{}&{\widehat{\bar C}/\adic^{n}\widehat{\bar C}}\\
    {}&{\widehat{\bar C}}\ar[ru]&{}\\
}
$$
\end{theorem}
\begin{proof}

We fix some generators of $B$ and get the following exact sequences
\begin{align*}
    0\to J\to &A[x]\stackrel{\beta}{\longrightarrow} B\to 0\\
    0\to {\bar J} \to &A[x]\stackrel{\bar\beta}{\longrightarrow} \bar B\to 0
\end{align*}
where $x=\{x_1,\ldots,x_n\}$ is a finite set of indeterminates.

Now, we replace $B$ by $S_{B}(J/J^2)$ and replace $\phi$ by the
composition of the projection $S_{B}(J/J^2)\to B$ (sending $J/J^2$
to zero) and $\phi$. So, we may assume that $(\Omega_{B/A})_t$ is
free $B_t$-module of rank $d$. Adding the generators $y$ of $J/J^2$
to the generators of $B$, we get a surjective homomorphism
$$
\beta'\colon A[x,y]\to S_{B}(J/J^2).
$$
Let $t=\{t_1,\ldots,t_d\}$ be the set of indeterminates, we replace
$\beta$ by the composition of the projection
$$
A[x,y,t]\to A[x,y],\quad t\mapsto 0
$$
and $\beta'$. The homomorphism $\bar \beta$ will be replaced by the
composition of the projection
$$
A[x,y,t]\to A[x],\quad y,t\mapsto 0
$$
and $\bar\beta$. Now, $J$ is the kernel of the new homomorphism
$\beta$. By Lemma~3 of~\cite{Elkik73}, we see that $(J/J^2)_t$ is a
free $B_t$-module.

Let the elements $(f_1,\ldots,f_q)\subseteq J$ form a basis of
$(J/J^2)_t$. Then there is an element $e\in J$ such that
$$
e^2-t^he\in (f_1,\ldots,f_q)
$$
for some natural $h$. Since the sequence
$$
0\to (J/J^2)_t \to \Omega_{A[x,y,t]/A}\otimes_{A[x,y,t]} B_t\to
\Omega_{B_t/A} \to 0
$$
is split exact, the ideal
$$
\triangle^{q}(f_1,\ldots,f_q)\subseteq A[x,y,t]
$$
generated by $q$-minors of the Jacobian matrix of $f_1,\ldots,f_q$
contains some $t^{h'}$. Enlarging $h'$ or replacing $e$ by
$t^{h-h'}e$, we may assume that $h=h'$.

By the Artin-Rees lemma, there are natural numbers $r$ and $c$ such
that
\begin{align*}
    \adic^r C&\cap \annihil_C (t) = 0\\
    \adic^m C&\cap {I} = \adic^{m-c}(\adic^c C\cap {I})\quad \mbox{whenever}\quad m\geqslant c
\end{align*}
Now, we define the ideal
$$
J_0 = (f_1,\ldots,f_q)\subseteq A[x,y,t]
$$
and the element $s=(t^h-e)t^h$. Since
$$
J_t = (f_1,\ldots,f_q,e)_t \quad \mbox{ and } \quad e^2-t^h e \in
(f_1,\ldots,f_q),
$$
we have $J_{s} = (f_1,\ldots,f_q)_s$. Therefore, $s^p J\subseteq
J_0$ for some natural $p$.

Now, we fix a natural number $n>\max(2h+c,h+r)$ and take the
reductions of $\varepsilon$ and $\bar\varepsilon$ modulo $\adic ^n
C$ and $\adic^n \bar C$, respectively:
$$
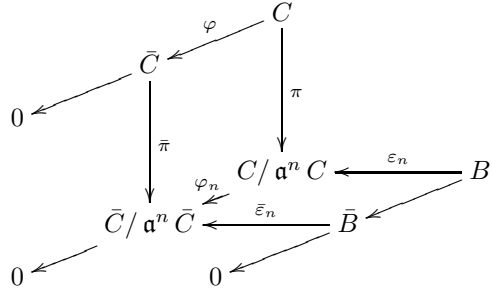
\xymatrix@!0@R=20pt@C=25pt{
    &&&&{C}\ar[dll]_{\varphi}\ar[ddd]^{\pi}&&&\\
    &&{\bar C}\ar[dll]\ar[ddd]^{\bar \pi}&&&&&\\
    0&&&&&&&\\
    &&&&{C/\adic ^n   C}\ar[dll]_{\varphi_n}&&&{B}\ar[dll]\ar[lll]_-{\varepsilon_n}\\
    &&{\bar C/\adic ^n   \bar C}\ar[dll]&&&{\bar B}\ar[dll]\ar[lll]_-{\bar \varepsilon_n}&&\\
    0&&&0&&&&\\
}
$$
Here $\pi$ and $\bar\pi$ are the quotient maps, $\varphi_n$ is
induced by $\varphi$, and the homomorphisms $\varepsilon_n$ and
$\bar\varepsilon_n$ are the reductions of $\varepsilon$ and $\bar
\varepsilon$, respectively.

Now, by Theorem~2 bis of~\cite{Elkik73}, the homomorphism $\bar
\varepsilon$ approximates by some $\bar \gamma\colon \bar B \to \bar
C$ module $\adic^n\bar C$. We denote the composition
$\varepsilon_n\circ \beta$ by $\alpha$ and let $\gamma'$ be an
arbitrary lifting of $\bar \gamma \circ \bar\beta$. We draw
everything on the following diagram:
$$
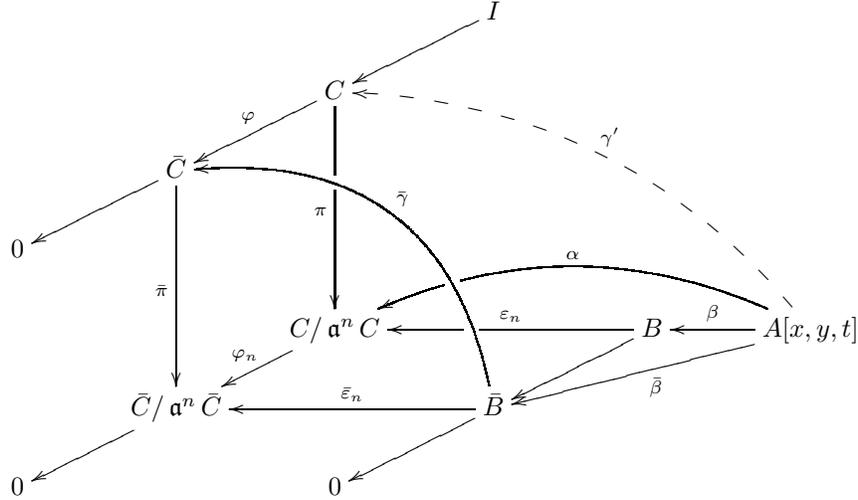
\xymatrix@!0@R=30pt@C=30pt{
    &&&&&&{{I}}\ar[dll]&&&&\\
    &&&&{C}\ar[dll]_{\varphi}\ar[ddd]_{\pi}|(0.38)\hole&&&&&&\\
    &&{\bar C}\ar[dll]\ar[ddd]_{\bar\pi}&&&&&&&&\\
    {0}&&&&&&&&&&\\
    &&&&{C/\adic ^{n}   C}\ar[dll]_{\varphi_n}&&&&{B}\ar[dll]\ar[llll]_-{\varepsilon_n} |(0.57)\hole&&{A[x,y,t]}\ar[ll]_-{\beta}\ar[lllld]^{\bar\beta}\ar@/_2pc/[llllll]_{\alpha} |(0.753)\hole\ar@{-->}@/_2pc/[lllllluuu]_{\gamma'}\\
    &&{\bar C/\adic ^{n}   \bar C}\ar[dll]&&&&{\bar B}\ar[dll]\ar[llll]_-{\bar\varepsilon_n}\ar@/_3pc/[lllluuu]_{\bar\gamma}&&&&\\
    {0}&&&&{0}&&&&&&\\
}
$$
The map $\gamma'$ is denoted by the dotted line because the triangle
with $\gamma'$, $\alpha$, and $\pi$ is not necessarily commutative.
Firstly, we are going to replace $\gamma'$ by $\gamma''$ such that
$\pi\circ \gamma'' = \alpha$. Secondly, we will show how to replace
$\gamma''$ by $\gamma$ such that $\gamma(J)=0$.

Let us define the following vector
$$
k'=\pi\circ\gamma'(x,y,t)-\alpha(x,y,t).
$$
If $k'=(k_i')$, then $k_i'\in \ker\varphi_n$. Therefore, we can find
$k=(k_i)$ with $k_i\in I$ and such that $\pi(k_i)=k_i'$. Let us
define the following homomorphism of $A$-algebras
$$
\gamma''\colon A[x,y,t]\to C\quad \mbox{ where} \quad
(x,y,t)\mapsto\gamma'(x,y,t)-k.
$$
Since $\gamma''$ coincides with $\gamma'$ modulo ${I}$ the whole
diagram with $\gamma''$ instead of $\gamma'$ is commutative. Now, we
are going to show that $\gamma(J)=0$. We will do this in two steps.

Step 1. We will show that there is a homomorphism $\gamma\colon
A[x,y,t]\to C$ such that $\gamma(J_0)=0$. Let $a=\gamma''(x,y,t)$ be
a vector of elements of $C$. We will search the homomorphism
$\gamma$ among ones sending $(x,y,t)$ to $a+t^hp$ for some vector
$p=(p_i)$, where $p_i\in \adic^{n-2h}C\cap{I}$. We also denote
$f=(f_i)$, the vector of generators of $J_0$.

By construction, we have
$$
\gamma''(J_0)\subseteq \gamma''(J)\subseteq \adic^nC\cap {I} =
\adic^{n-c}(\adic^c C\cap {I})
$$
Thus $f(a)=t^{2h}e$, where $e=(e_i)$ and $e_i\in \adic^{n-2h}C\cap
{I}$. Since $t^h\in\triangle^{q}(f)$, we have $M(a)N = t^hE$, where
$M(a)$ is the Jacobian matrix of $f$ at the point $a$, $N$ is some
matrix and $E$ is the identity matrix. Now, we want to have
$$
0=f(a+tp) = f(a) + t^hM(a)p + t^{2h}p^tQp.
$$
where the right-hand part is the Taylor expansion up to the second
order term and $p^t$ is the transposed vector $p$. We replace $f(a)$
by $t^{2h}e$ and $t^hE$ by $M(a)N$. As the result, we will obtain
the following equation
$$
0= t^hM(a) (p + N(e + p^tQp))
$$
It is enough to find $p\in \adic^{n-2h}C\cap {I}$ such that
$$
p + N(e + p^tQp) = 0.
$$
The zero vector $p=0$ is a solution of this system in
$C/\adic^{n-2h}C$ and the Jacobian matrix of this system at $p=0$ is
the identity matrix $E$. Since $(C,\adic C)$ is henselian, there is
a solution $p$ in $C$ such that $p_i\in \adic^{n-2h}C$. But then
$$
(E+p^{t}Q)p = -Ne,\quad \mbox{ and } \quad e_i\in \adic^{n-2h}C\cap
{I}
$$
Since all $p_i$ are in the radical of $C$, the matrix $(E+p^{t}Q)$
is invertible and, hence,  $p_i\in \adic^{n-2h}C\cap {I}$.

Step 2. We will show that $\gamma(J)=0$. By construction, the
completion of the map
$$
\gamma\colon A[x,y,t]\to C\quad\mbox{ by }\quad(x,y,t)\mapsto a+t^hp
$$
coincides with $\varepsilon$ modulo $\adic^{n-h}\widehat{C}$. Thus
$\gamma(J)\subseteq \adic^{n-h}C$. Moreover, by definition of $s$,
we have
$$
t^{2hp}\gamma(J)=\gamma(s^pJ)\subseteq \gamma(J_0) = 0
$$
Then by the choice of $n$ and $r$, it follows that $\gamma(J)=0$.

We have constructed $\gamma\colon A[x,y,t]\to C$ being a lifting of
$\varepsilon_{n-h}\circ \beta$. It induces a homomorphism of
$\gamma\colon B\to C$ being a lifting of $\varepsilon_{n-h}$. Since
$\bar \gamma$ is a lifting of $\bar\varepsilon_n$, it is also a
lifting of $\bar \varepsilon_{n-h}$. Since $n$ is an arbitrary
large, we are done.
\end{proof}

\begin{corollary}\label{cor:duo_alg_lift}

Let $A$ be a ring $\adic=(t)\subseteq A$ a principal ideal. Let
$\phi\colon B\to \bar B$ be a surjective $A$-homomorphism of
finitely generated $A$-algebras, $\varphi\colon C\to \bar C$ is a
homomorphism of $A$-algebras. Assume that  $B$ and $\bar B$ are
smooth over the complement of $V(\adic)$ and we are given
homomorphisms $\varepsilon \colon \widehat{B}\to \widehat{C}$ and
$\bar \varepsilon \colon \widehat{\bar B}\to \widehat{\bar C}$ such
that the following diagram with the solid arrows only is commutative
$$
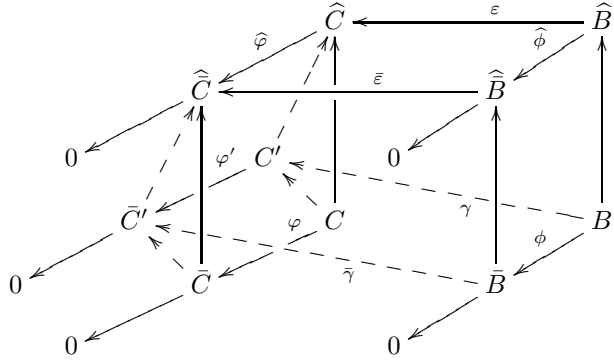
\xymatrix@R=10pt@C=10pt{
    &&&&&{\widehat{C}}\ar[dll]_{\widehat{\varphi}}&&&&&{\widehat{B}}\ar[dll]_{\widehat{\phi}}\ar[lllll]_(0.4){\varepsilon}\\
    &&&{\widehat{\bar C}}\ar[dll]&&&&&{\widehat{\bar B}}\ar[dll]\ar[lllll]_(0.4){\bar\varepsilon}&&\\
    &0&&&{C'}\ar[dll]_(0.25){\varphi'}|!{[l];[lu]}\hole\ar@{-->}[uur]|!{[u];[ur]}\hole&&0&&&&\\
    &&{\bar C'}\ar[dll]\ar@{-->}[uur]&&&{C}\ar@{-->}[lu]\ar[dll]_(0.25){\varphi}|!{[lll];[rrrd]}\hole\ar[uuu]|!{[uu];[uur]}\hole&&&&&{B}\ar[dll]_{\phi}\ar[uuu]\ar@{-->}[ullllll]^(0.4){\gamma}|!{[l];[lu]}\hole\\
    0&&&{\bar C}\ar@{-->}[lu]\ar[dll]\ar[uuu]&&&&&{\bar B}\ar[dll]\ar[uuu]\ar@{-->}[ullllll]^(0.4){\bar\gamma}&&\\
    &0&&&&&0&&&&\\
}
$$
Then, for every natural number $n$, there exist \'etale extensions
$C\to C'$ and $\bar C\to\bar C'$ such that the induced by $\varphi$
map $\varphi'\colon C'\to \bar C'$ is well-defined and surjective.
There also exist two homomorphisms $\gamma\colon B\to C$ and $\bar
\gamma \colon \bar B\to \bar C$ such that the squares with
dotted arrows on the previous diagram is commutative, and the
following diagrams are commutative
$$
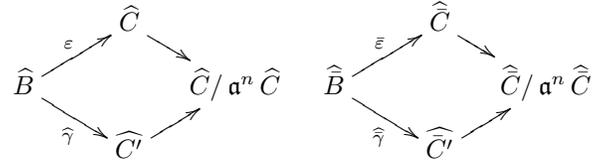
\xymatrix@!0@C=40pt{
    {}&{\widehat{C}}\ar[rd]&{}\\
    {\widehat{B}}\ar[ru]^{\varepsilon}\ar[rd]_{\widehat{\gamma}}&{}&{\widehat{C}/\adic^{n}\widehat{C}}\\
    {}&{\widehat{C'}}\ar[ru]&{}\\
}
 \quad
\xymatrix@!0@C=40pt{
    {}&{\widehat{\bar C}}\ar[rd]&{}\\
    {\widehat{\bar B}}\ar[ru]^{\bar \varepsilon}\ar[rd]_{\widehat{\bar\gamma}}&{}&{\widehat{\bar C}/\adic^{n}\widehat{\bar C}}\\
    {}&{\widehat{\bar C'}}\ar[ru]&{}\\
}
$$
\end{corollary}

\begin{lemma}\label{lemma:hausdorff_quot}
Let $A$ be a ring, $\adic\subseteq A$, $A$ is $\adic$-adically
complete, and $M\subseteq P$ are finitely generated $A$-modules.
Then $M = \bigcap_{k\geqslant 0} M+\adic^k P$.
\end{lemma}
\begin{proof}
Since $A$ is complete, the module $P/M$ is complete. In particular,
$P/M$ is separated. Thus $\bigcap_{k\geqslant 0}\adic^k P/M = 0$.
Whence $M = \bigcap_{k\geqslant 0} M+\adic^k P$.
\end{proof}

The following corollary is a direct generalization of Theorem~2~bis of~\cite{Elkik73}. The difference between the corollary and Theorem~\ref{theorem:com_alg_lift_h} is the additional condition on the sections. They must factors through given open subsets.

\begin{corollary}
Let $A$ be a ring, $\adic=(t)\subseteq A$ a principal ideal,
$$
\phi\colon B\to \bar B \quad \mbox{ and }\quad \psi\colon C\to \bar
C
$$
be surjective homomorphisms of $A$-algebras such that $B$ and $\bar
B$ are finitely generated, and the pairs $(C,\adic C)$ and $(\bar
C,\adic \bar C)$ are henselian. Let we be given ideals $I\subseteq
B$, $\bar I\subseteq \bar B$ and $A$-homomorphisms
$$
\varphi\colon B\to
\widehat{C}\quad \mbox{ and }\quad \bar \varphi\colon \bar B\to
\widehat{\bar C}
$$
such that
$$
\varphi(I)\widehat{C}\supseteq \adic^c\quad \mbox{ and
}\quad\bar\varphi(\bar I)\widehat{\bar C}\supseteq \adic^c
$$
for some natural $c$. Then, for every $n$, there exist
$A$-homomorphisms
$$
\gamma\colon B\to C\quad \mbox{ and }\quad \bar \gamma \colon \bar
B\to \bar C
$$
such that
$$
\gamma(I)C\supseteq\adic^c\quad \mbox{ and }\quad \bar \gamma(\bar
I)\bar C\supseteq \adic^c
$$
and the following diagram is commutative
$$
\xymatrix@!0@R=30pt@C=60pt{
    {}&{C}\ar[dd]|!{[dl];[d]}\hole\ar[dl]_{\psi}&{}&{B}\ar[dd]^(0.4){\varphi}\ar[dl]_{\phi}\ar@{-->}[ll]_(0.3){\gamma}\\
    {\bar C}\ar[dd]&{}&{\bar B}\ar[dd]^(0.4){\bar\varphi}\ar@{-->}[ll]_(0.3){\bar \gamma}&{}\\
    {}&{C/\adic^n C}\ar[dl]_{\psi_n}&{}&{\widehat{C}}\ar[dl]_{\widehat{\psi}}\ar[ll]|!{[lu];[l]}\hole\\
    {\bar C/\adic^n\bar C}&{}&{\widehat{\bar C}}\ar[ll]&{}\\
}
$$
\end{corollary}
\begin{proof}
We fix $m >\max(n,c)$ and apply
Theorem~\ref{theorem:com_alg_lift_h} with $m$ instead of $n$.
Therefore, we find $A$-homomorphisms $\gamma$ and $\bar \gamma$ as
on the previous diagram. We will show that $\adic^c\subseteq
\gamma(I)C$, the second inclusion is proven in the same way. By
definition, $\varphi(I)\subseteq \gamma(I)+\adic^m\widehat{C}$.
Hence $(\adic \widehat{C})^c\subseteq \gamma(I)\widehat{C}+(\adic
\widehat{C})^m$. Thus, by induction, we have $(\adic
\widehat{C})^c\subseteq \gamma(I)\widehat{C}+(\adic \widehat{C})^k$
for all $k\geqslant m$. By Lemma~\ref{lemma:hausdorff_quot},
$\adic^c \widehat{C}\subseteq \gamma(I)\widehat{C}$. Since $C$ is
henselian, $C\to \widehat{C}$ is faithfully flat. Therefore, for any
ideal $J\subseteq C$, we have $(J \widehat{C})\cap C = J$. In
particular, $\adic^c C \subseteq \gamma(I)C$.

\end{proof}

\section{Approximation of pairs}

\subsection{Approximation of henselizations}

\begin{theorem}\label{theorem:pair_lift}
Let $A$ be a ring $\adic=(t)\subseteq A$ is a principal ideal,
$\phi\colon B\to \bar B$ and $\psi\colon C\to \bar C$ are surjective
$A$-homomorphisms of $A$-algebras such that
\begin{enumerate}
\item $B$ and $\bar B$ are finitely generated over $A$.

\item $(C,\adic C)$ and $(\bar C,\adic \bar C)$ are henselian.

\item $I$ is the kernel of $\psi$.
\end{enumerate}
Let $n,r,h,c$ be natural numbers such that
\begin{enumerate}
\item $\adic^r C\cap \annihil_C(t)=0$ and $\adic^r\bar C\cap \annihil_{\bar
C}(t)=0$.

\item $\adic^kC\cap I = \adic^{k-c}(\adic^c C\cap I)$ for all $k\geqslant
c$.

\item $\adic^h\subseteq  H_{B/A}$ and $\adic^h\subseteq  H_{\bar
B/A}$.

\item $n>\max(r+2h, c+2h, 4h)$.
\end{enumerate}
and we are given a commutative diagram
$$
\xymatrix{
    {B}\ar[r]^{\phi}\ar[d]^{\alpha_n}&{\bar B}\ar[d]^{\beta_n}\\
    {C/\adic^n C}\ar[r]^{\psi_n}&{\bar C/\adic^n \bar C}\\
}
$$
Then there exist $A$-homomorphisms $\alpha\colon B\to C$ and
$\beta\colon \bar B\to \bar C$ such that the following diagram is
commutative
$$
\xymatrix@!0@R=30pt@C=60pt{
    {}&{C}\ar[dd]|!{[d];[dl]}\hole\ar[dl]_{\psi}&{}&{B}\ar[dd]^(0.4){\alpha_n}\ar[dl]_{\phi}\ar@{-->}[ll]_(0.3){\alpha}\\
    {\bar C}\ar[dd]&{}&{\bar B}\ar[dd]^(0.4){\beta_n}\ar@{-->}[ll]_(0.3){\beta}&{}\\
    {}&{C/\adic^{n-2h}C}\ar[dl]_{\psi_{n-2h}}&{}&{C/\adic^nC}\ar[dl]_{\psi_n}\ar[ll]|!{[l];[lu]}\hole\\
    {\bar C/\adic^{n-2h}\bar C}&{}&{\bar C/\adic^n \bar C}\ar[ll]&{}\\
}
$$
\end{theorem}
\begin{proof}
The assertion immediately follows from
Theorem~\ref{theorem:common_alg_lifting} and
Theorem~\ref{theorem:com_alg_lift_h}.
\end{proof}

\begin{corollary}\label{cor:pair_lift}
Let $A$ be a ring $\adic=(t)\subseteq A$ is a principal ideal,
$\phi\colon B\to \bar B$ and $\psi\colon C\to \bar C$ are surjective
$A$-homomorphisms of finitely generated $A$-algebras, and  $I$ is the kernel of $\psi$. Let $n,r,h,c$ be natural numbers such that
\begin{enumerate}
\item $\adic^r C\cap \annihil_C(t)=0$ and $\adic^r\bar C\cap \annihil_{\bar
C}(t)=0$.

\item $\adic^kC\cap I = \adic^{k-c}(\adic^c C\cap I)$ for all $k\geqslant
c$.

\item $\adic^h\subseteq  H_{B/A}$ and $\adic^h\subseteq  H_{\bar
B/A}$.

\item $n>\max(r+2h, c+2h, 4h)$.
\end{enumerate}
and we are given a commutative diagram
$$
\xymatrix{
    {B}\ar[r]^{\phi}\ar[d]^{\alpha_n}&{\bar B}\ar[d]^{\beta_n}\\
    {C/\adic^n C}\ar[r]^{\psi_n}&{\bar C/\adic^n \bar C}\\
}
$$
Then there exist strict etale extensions $C\to C'$, $\bar C\to \bar C'$, and a surjective $A$-homomorphism $\psi'\colon C'\to \bar C'$ together with $A$-homomorphisms $\alpha\colon B\to C'$ and $\beta\colon \bar B\to \bar C'$ such that the following diagram is commutative
$$
\xymatrix@!0@R=30pt@C=60pt{
    {}&{C'}\ar[dd]|!{[d];[dl]}\hole\ar[dl]_{\psi'}&{}&{B}\ar[dd]^(0.4){\alpha_n}\ar[dl]_{\phi}\ar@{-->}[ll]_(0.3){\alpha}\\
    {\bar C'}\ar[dd]&{}&{\bar B}\ar[dd]^(0.4){\beta_n}\ar@{-->}[ll]_(0.3){\beta}&{}\\
    {}&{C/\adic^{n-2h}C}\ar[dl]_{\psi_{n-2h}}&{}&{C/\adic^nC}\ar[dl]_{\psi_n}\ar[ll]|!{[l];[lu]}\hole\\
    {\bar C/\adic^{n-2h}\bar C}&{}&{\bar C/\adic^n \bar C}\ar[ll]&{}\\
}
$$
\end{corollary}

\begin{lemma}\label{lemma:adic_smooth_fbr}
Let $A$ be a ring, $\adic=(t)\subseteq A$ an ideal, $\alpha\colon A\to B$ is a homomorphism such that $B/\adic^k B$ is a finitely generated $A/\adic^k$-algebra for all $k$. Let we be given numbers $n>r$ such that
\begin{enumerate}
\item $\adic^rB\cap \annihil_B(t)=0$.
\item The map $\alpha_n\colon A/\adic^n\to B/\adic^n B$ is smooth.
\end{enumerate}
Assume that $\adic B$ belongs to the Jacobson radical of $B$, then all the maps $\alpha_k$ are smooth.
\end{lemma}
\begin{proof}
By Proposition~17.5.2 of~\cite{EGA1967_32}, it is enough to show that all the maps $\alpha_k$ are flat. Since $\alpha_n$ is flat, the maps
$$
\adic^k/\adic^{k+1}\otimes_{A/\adic} B/\adic B \to \adic^k B/\adic^{k+1}B
$$
are isomorphisms for all $k<n$ by Theorem~49 of~\cite{MatsumCA}. Because of the choice of $r$, the maps
$$
\adic^r /\adic^{r+1}\stackrel{t^{k-r}}{\longrightarrow}\adic^k/\adic^{k+1}
\quad \mbox{ and } \quad
\adic^r B/\adic^{r+1}B\stackrel{t^{k-r}}{\longrightarrow} \adic^kB/\adic^{k+1}B
$$
are isomorphisms. Hence,
$$
\adic^k/\adic^{k+1}\otimes_{A/\adic} B/\adic B \to \adic^k B/\adic^{k+1}B
$$
holds for all $k$ and, by Theorem~49 of~\cite{MatsumCA}, the maps $\alpha_k$ are flat for all $k$.
\end{proof}

\begin{lemma}\label{lemma:adic_smooth_eq_smooth}
Let $A$ and $B$ be rings, $\adic\subseteq A$ an ideal, $\alpha\colon A\to B$ is a homomorphism such that $B$ becomes a localisation of a finitely generated $A$-algebra. Suppose also that $\adic B$ belongs to the Jacobson radical of $B$ and $\alpha_k\colon A/\adic^k \to B/\adic^k B$ are smooth. Then $\alpha$ is smooth.
\end{lemma}
\begin{proof}
Since $B$ is a localisation of a finitely generated algebra, there is an exact sequence of $A$-homomorphisms
$$
0\to J\to R\to B\to 0
$$
where $R$ is smooth over $A$. The condition on $\alpha_k$ implies that $B$ is smooth over $A$ in $\adic B$-topology (see~\cite{MatsumCA}). Thus, for all $n$, the homomorphism of $B/\adic^n B$-modules
$$
J/J^2\otimes_B B/\adic^n B\to \Omega_{R/A}\otimes_R B/\adic^n B
$$
is left-invertible by Theorem~63 of~\cite{MatsumCA}. Since $\Omega_{R/A}\otimes_R B$ is projective and $\adic$ belongs to the Jacobson radical of $B$, the homomorphism
$$
J/J^2\to \Omega_{R/A}\otimes_R B
$$
is left-invertible by Lemma~2 of~\cite[(29.A), p.~215]{MatsumCA}. Hence, again by Theorem~63, $B$ is smooth over $A$.
\end{proof}

\begin{lemma}\label{lemma:smooth_lifting}
Let $A$ be a ring $\adic=(t)\subseteq A$ is a principal ideal,
$\phi\colon B\to \bar B$ and $\psi\colon C\to \bar C$ are surjective
$A$-homomorphisms of finitely generated $A$-algebras, and  $I$ is the kernel of $\psi$. Let $n,r,h,c$ be natural numbers such that
\begin{enumerate}
\item $\adic^r C\cap \annihil_C(t)=0$ and $\adic^r\bar C\cap \annihil_{\bar
C}(t)=0$.

\item $\adic^kC\cap I = \adic^{k-c}(\adic^c C\cap I)$ for all $k\geqslant
c$.

\item $\adic^h\subseteq  H_{B/A}$ and $\adic^h\subseteq  H_{\bar
B/A}$.

\item $n>\max(r+2h, c+2h, 4h)$.
\end{enumerate}
and we are given a commutative diagram
$$
\xymatrix{
    {B/\adic^nB}\ar[r]^{\phi_n}\ar[d]^{\alpha_n}&{\bar B/\adic^n \bar B}\ar[d]^{\beta_n}\\
    {C/\adic^n C}\ar[r]^{\psi_n}&{\bar C/\adic^n \bar C}\\
}
$$
where $\alpha_n$ and $\beta_n$ are smooth.
Then there exist strict etale extensions $C\to C'$, $\bar C\to \bar C'$, and a surjective $A$-homomorphism $\psi'\colon C'\to \bar C'$ together with smooth $A$-homomorphisms $\alpha\colon B\to C'$ and $\beta\colon \bar B\to \bar C'$ such that the following diagram is commutative
$$
\xymatrix@!0@R=30pt@C=60pt{
    {}&{C'}\ar[dd]|!{[d];[dl]}\hole\ar[dl]_{\psi'}&{}&{B}\ar[dd]^(0.4){\alpha_n}\ar[dl]_{\phi}\ar@{-->}[ll]_(0.3){\alpha}\\
    {\bar C'}\ar[dd]&{}&{\bar B}\ar[dd]^(0.4){\beta_n}\ar@{-->}[ll]_(0.3){\beta}&{}\\
    {}&{C/\adic^{n-2h}C}\ar[dl]_{\psi_{n-2h}}&{}&{C/\adic^nC}\ar[dl]_{\psi_n}\ar[ll]|!{[l];[lu]}\hole\\
    {\bar C/\adic^{n-2h}\bar C}&{}&{\bar C/\adic^n \bar C}\ar[ll]&{}\\
}
$$
\end{lemma}
\begin{proof}
By Corollary~\ref{cor:pair_lift}, we find the homomorphisms $\alpha$ and $\beta$. We will find an element $s\in 1+\adic C'$ such that the compositions
$$
B\stackrel{\alpha}{\longrightarrow} C'\to C'_s
\quad \mbox{ and }\quad
\bar B\stackrel{\beta}{\longrightarrow} \bar C'\to \bar C'_s
$$
are smooth and this will prove the lemma.

Let us consider the case of $\alpha$. We also denote $C'$ by $C$ and $\psi'$ by $\psi$ for short. Let $S = 1 + \adic B$ and $T = 1 + \adic C$, then the induced homomorphism $S^{-1}B\to T^{-1}C$ satisfies the conditions of Lemma~\ref{lemma:adic_smooth_fbr}. Hence, all $\alpha_k\colon B/\adic^k B\to C/\adic^k C$ are smooth. Thus, $T^{-1}C$ is smooth over $S^{-1}B$ by Lemma~\ref{lemma:adic_smooth_eq_smooth}. Consequently, $T^{-1}C$ is $B$ smooth. Now, we see that
$$
T^{-1}(H_{C/B}) = H_{(T^{-1}C)/B} = (1)
$$
by Lemma~\ref{lemma:HB_smooth}. Thus, there is an element $t\in T$, such that
$$
H_{C_t/B} = (H_{C/B})_t = (1).
$$
Hence, $C_t$ is smooth over $B$.

If $\bar t\in 1+\adic \bar C$ denote an element such that $\bar C_{\bar t}$ is $\bar B$-smooth, then we lift it to an element $\bar t'\in 1+\adic C$. Whence, $s=t\bar t'$ is the required element.
\end{proof}

\begin{lemma}\label{lemma:isom_lift_compl}
Let $A$ be a ring, $\adic\subseteq A$ and ideal, and $\varphi\colon A\to A$ is a homomorphism such that $A$ is $\adic$-adically complete, $\varphi$ takes $\adic$ to $\adic$ and induces the identity map on $A/\adic$. Then $\varphi$ is an isomorphism
\end{lemma}
\begin{proof}
By definition, $\varphi(x)=x+\gamma(x)$, where $\gamma\colon A\to \adic$ is a linear map such that $\gamma(xy)=x\gamma(y)+\gamma(x)y+\gamma(x)\gamma(y)$. Hence, for any $x\in \adic$, $\gamma^n(x)\subseteq \adic^n$. Thus the map
$$
\psi(x)=\sum_{n=0}^{\infty}(-1)^{n}\gamma^n(x)
$$
is well-defined and is inverse to $\varphi$.
\end{proof}

\begin{proposition}\label{prop:app_hens}
Let $A$ be a ring $\adic=(t)\subseteq A$ is a principal ideal,
$\phi\colon B\to \bar B$ and $\psi\colon C\to \bar C$ are surjective
$A$-homomorphisms of finitely generated $A$-algebras, and $I$ is the kernel of $\psi$.

Let $n,r,h,c$ be natural numbers such that
\begin{enumerate}
\item $\adic^r C\cap \annihil_C(t)=0$ and $\adic^r\bar C\cap \annihil_{\bar
C}(t)=0$.

\item $\adic^kC\cap I = \adic^{k-c}(\adic^c C\cap I)$ for all $k\geqslant
c$.

\item $\adic^h\subseteq  H_{B/A}$ and $\adic^h\subseteq  H_{\bar
B/A}$.

\item $n>\max(r+2h, c+2h, 4h)$.
\end{enumerate}
and we are given a commutative diagram
$$
\xymatrix{
    {B/\adic^nB}\ar[r]^{\phi_n}\ar[d]^{\alpha_n}&{\bar B/\adic^n \bar B}\ar[d]^{\beta_n}\\
    {C/\adic^n C}\ar[r]^{\psi_n}&{\bar C/\adic^n \bar C}\\
}
$$
where $\alpha_n$ and $\beta_n$ are isomorphisms.
Then there exist $A$-isomorphisms $\alpha^h\colon B^h\to C^h$ and $\beta^h\colon \bar B^h\to \bar C^h$ such that the following diagram is commutative
$$
\xymatrix@!0@R=30pt@C=60pt{
    {}&{C^h}\ar[dd]|!{[d];[dl]}\hole\ar[dl]_{\psi^h}&{}&{B^h}\ar[dd]^(0.4){\alpha_n}\ar[dl]_{\phi^h}\ar@{-->}[ll]_(0.3){\alpha^h}\\
    {\bar C^h}\ar[dd]&{}&{\bar B^h}\ar[dd]^(0.4){\beta_n}\ar@{-->}[ll]_(0.3){\beta^h}&{}\\
    {}&{C/\adic^{n-2h}C}\ar[dl]_{\psi_{n-2h}}&{}&{C/\adic^nC}\ar[dl]_{\psi_n}\ar[ll]|!{[l];[lu]}\hole\\
    {\bar C/\adic^{n-2h}\bar C}&{}&{\bar C/\adic^n \bar C}\ar[ll]&{}\\
}
$$
\end{proposition}
\begin{proof}
From Lemma~\ref{lemma:smooth_lifting} it follows that, replacing $C$ and $\bar C$ by their strict etale extensions, we may suppose that we have a pair of smooth homomorphisms $\alpha \colon B\to C$ and $\beta\colon \bar B\to \bar C$ such that the following diagram is commutative
$$
\xymatrix@!0@R=30pt@C=60pt{
    {}&{C}\ar[dd]|!{[d];[dl]}\hole\ar[dl]_{\psi}&{}&{B}\ar[dd]\ar[dl]_{\phi}\ar[ll]_(0.3){\alpha}\\
    {\bar C}\ar[dd]&{}&{\bar B}\ar[dd]\ar[ll]_(0.3){\beta}&{}\\
    {}&{C/\adic^{n-2h}C}\ar[dl]_{\psi_{n-2h}}&{}&{B/\adic^{n-2h}B}\ar[dl]_{\psi_n}\ar[ll]|!{[l];[lu]}\hole_(0.35){\alpha_{n-2h}}\\
    {\bar C/\adic^{n-2h}\bar C}&{}&{\bar B/\adic^{n-2h} \bar B}\ar[ll]_(0.35){\beta_{n-2h}}&{}\\
}
$$
Now, we should show that $\alpha$ and $\beta$ induce isomorphisms of henselizations. We will demonstrate this for $\alpha$.

Firstly, note that all $\alpha_n$ are isomorphisms by Lemma~\ref{lemma:isom_lift_compl}. Then, we may tensor by $B^h$ over $B$. Hence, we may assume that $(B,\adic B)$ is henselian and $C$ is a smooth finitely generated $B$-algebra. Thus, by Theorem~2 of~\cite{Elkik73}, the composition
$$
C\to C/\adic^{n-2h}C \stackrel{\alpha_{n-2h}^{-1}}{\longrightarrow} B/\adic^{n-2h}B
$$
lifts to a $B$-homomorphism $\varepsilon \colon C\to B$. We will show that the induced map $\varepsilon^h\colon C^h\to B^h=B$ is an isomorphism. The surjectivity is clear. Since $C^h$ is a subalgebra of $\widehat{C}$, it is enough to show that $\widehat{\varepsilon}\colon \widehat{C}\to \widehat{B}$ is injective. But
$$
(\widehat{\alpha}\circ\widehat{\varepsilon})_{n-2h} = \alpha_{n-2h}\circ\varepsilon_{n-2h} = Id
$$
Therefore, by Lemma~\ref{lemma:isom_lift_compl}, $\widehat{\alpha}\circ\widehat{\varepsilon}$ is an isomorphism and we are done.
\end{proof}

\subsection{Approximation of complete algebras}

\begin{lemma}\label{lemma:surj_quot}
Let $A$ be a ring, $\adic\subseteq A$ an ideal, and $\varphi \colon
A\to B$ is a homomorphism such that $A$ is $\adic$-adically
complete, $B$ is $\adic B$-adically separated, and the map $\bar
\varphi \colon A/\adic \to B/\adic B$ induced by $\varphi$ is
surjective. Then the map $\varphi$ is surjective.
\end{lemma}
\begin{proof}
Since $B = \Im \varphi +\adic B$, for every $b\in B$, there exist
$a_1\in A$, $s_1\in \adic$, and $b_1\in B$ such that $b=\varphi(a_1)
+ s_1b_1$. Applying the arguments for $b_1$ instead of $b$ we find
$a_1\in A$, $s_2\in \adic$, and $b_2\in B$ such that
$b_1=\varphi(a_2) + s_2b_2$. Hence
$$
b = \varphi(a_1+s_1a_2) + s_2b_2
$$
By induction, we find sequences $a_i\in A$, $s_i\in \adic$, and
$b_i\in B$. The sequence
$$
r_n = \sum_{k\leqslant n} \left( \prod_{i\leqslant k} s_i\right)a_k
$$
converges to an element $r\in A$ because $A$ is complete. Since $B$
is separated, the difference
$$
b-\varphi(r_n) = \prod_{k\leqslant n+1} s_k b_{n+1}
$$
tends to zero. Therefore, $b = \varphi(r)$.
\end{proof}

\begin{proposition}\label{prop:isom_lift}
Let $A$ be a ring $\adic=(t)\subseteq A$ a principal ideal such that
$A$ is $\adic$-adically complete. Let $\phi\colon B\to \bar B$ and
$\psi\colon B_0\to \bar B_0$ be surjective $A$-homomorphisms of
formally finitely generated $A$-algebras. Let $\varepsilon \colon
A\{X\}\to B$ be a surjective $A$-homomorphism, where $X =
\{x_1,\ldots, x_n\}$ is a finite set, and $J=\ker \varepsilon$,
$\bar J = \ker(\phi\circ \varepsilon)$. Assume that $l,r,c,h,n$ are
natural numbers such that
\begin{enumerate}
\item $\adic^r B\cap \annihil_B(t)=0$ and $\adic^r\bar B\cap \annihil_{\bar
B}(t)=0$.

\item $\adic^kB\cap I = \adic^{k-c}(\adic^c B\cap I)$ for all $k\geqslant
c$.

\item $\adic^h\subseteq  \bar H_{B/A}$ and $\adic^h\subseteq  \bar H_{\bar
B/A}$.

\item $\adic^{d}\{X\}\cap J = \adic^{d-l}(\adic^l\{X\}\cap J)$ and
$\adic^{d}\{X\}\cap \bar J = \adic^{d-l}(\adic^l\{X\}\cap \bar J)$
for all $d\geqslant l$.

\item $n>\max(r+2h, c+2h, l+2h, 4h)$.

\end{enumerate}
and there are isomorphisms $\alpha_n\colon B_0/\adic^{n}B_0\to
B/\adic^n B$ and $\beta_n\colon \bar B_0/\adic^n\bar B_0\to \bar
B/\adic^n \bar B$ such that the following diagram is commutative
$$
\xymatrix{
    {B_0/\adic^{n} B_0}\ar[r]^{\psi_n}\ar[d]^{\alpha_n}&{\bar B_0/\adic^{n}\bar B_0}\ar[d]^{\beta_n}\\
    {B/\adic^{n}B}\ar[r]^{\phi_n}&{\bar B/\adic^{n}\bar B}\\
}
$$
Then, there exist isomorphisms $\alpha\colon B_0\to B$ and
$\beta\colon \bar B_0\to \bar B$ such that the following diagram is
commutative
$$
\xymatrix@!0@R=40pt@C=70pt{
    {}&{B_0}\ar[rr]^(0.6){\psi}\ar@{-->}[dl]_{\alpha}\ar[dd]|!{[dl];[d]}\hole&{}&{\bar B_0}\ar[dd]\ar@{-->}[dl]_{\beta}\\
    {B}\ar[rr]^(0.6){\phi}\ar[dd]&{}&{\bar B}\ar[dd]&{}\\
    {}&{B_0/\adic^{n-2h}B_0}\ar[rr]^(0.4){\psi_{n-2h}}|!{[ru];[r]}\hole\ar[dl]_{\alpha_{n-2h}}&{}&{\bar B_0/\adic^{n-2h}\bar B_0}\ar[dl]_{\beta_{n-2h}}\\
    {B/\adic^{n-2h}B}\ar[rr]^(0.4){\phi_{n-2h}}&{}&{\bar B/\adic^{n-2h}\bar B}&{}\\
}
$$
\end{proposition}
\begin{proof}
By Theorem~\ref{theorem:common_alg_lifting}, there exist
$A$-homomorphisms $\alpha\colon B_0\to B$ and $\beta\colon \bar
B_0\to \bar B$ such that in the following diagram right cube is
commutative \small
$$
\xymatrix@!0@R=25pt@C=40pt{
    {}&{\bar J_0}\ar[rr]\ar[dl]\ar[dd]|!{[dl];[d]}\hole&{}&{J_0}\ar[rr]\ar[dl]\ar[dd]|!{[dl];[d]}\hole&{}&{A\{X\}}\ar@{-->}[rr]^{\varepsilon_0}\ar@{=}[dl]\ar[dd]|!{[dl];[d]}\hole&{}&{B_0}\ar[rr]^{\psi}\ar[dl]_{\alpha}\ar[dd]|!{[dl];[d]}\hole&{}&{\bar B_0}\ar[dl]_{\beta}\ar[dd]\\
    {\bar J}\ar[rr]\ar[dd]&{}&{J}\ar[rr]\ar[dd]&{}&{A\{X\}}\ar[rr]^(0.7){\varepsilon}\ar[dd]&{}&{B}\ar[rr]^(0.7){\phi}\ar[dd]&{}&{\bar B}\ar[dd]&{}\\
    {}&{\bar J_0/\bar J_0\cap \adic^m\{X\}}\ar[rr]|!{[ru];[r]}\hole\ar[dl]&{}&{J_0/J_0\cap \adic^m\{X\}}\ar[rr]|!{[ru];[r]}\hole\ar[dl]&{}&{A/\adic^m[X]}\ar[rr]|!{[ru];[r]}\hole\ar@{=}[dl]&{}&{B_0/\adic^m B_0}\ar[rr]^(0.4){\psi_m}|!{[ru];[r]}\hole\ar[dl]_{\alpha_m}&{}&{\bar B_0/\adic^m \bar B_0}\ar[dl]_{\beta_m}\\
    {\bar J/\bar J\cap \adic^m\{X\}}\ar[rr]&{}&{J/J\cap \adic^m\{X\}}\ar[rr]&{}&{A/\adic^m[X]}\ar[rr]&{}&{B/\adic^m B}\ar[rr]^{\phi_m}&{}&{\bar B/\adic^m \bar B}&{}\\
}
$$
\normalsize where $m=n-2h$. By Lemma~\ref{lemma:surj_quot}, the map
$\alpha$ is surjective, and we define $\varepsilon_0$ to be a
lifting of $\varepsilon$. Again, by Lemma~\ref{lemma:surj_quot}, the
map $\varepsilon_0$ is surjective, and we define $J_0=\ker
\varepsilon_0$ and $\bar J_0 = \ker(\psi \circ \varepsilon_0)$.
Therefore, we have constructed the diagram above such that it
becomes commutative. By construction, $J_0\subseteq J$ and $\bar
J_0\subseteq \bar J$. Since the left cube of the diagram is
commutative, we have
\begin{align*}
J &= J_0+ J\cap \adic^m\{X\} \subseteq J_0 + \adic J\\
\bar J &= \bar J_0 + \bar J\cap \adic^m\{X\} \subseteq \bar J_0
+\adic \bar J
\end{align*}
The inclusions above hold because of item~4 from the hypothesis of
the proposition. Hence $J=J_0$ and $\bar J= \bar J_0$ by Nakayama's lemma.
Therefore, the maps $\alpha$ and $\beta$ are isomorphisms.
\end{proof}

\section{Algebraization of pairs}

\subsection{Pairs of modules}

\begin{lemma}\label{lemma:duo_lift}
Let $A$ be a ring, $\adic=(t)\subseteq A$ a principal ideal, and
$B\to \bar B\to 0$ is a surjection of $A$-algebras. Let $M$ and $P$
be $B$-modules and ${\bar M}$ and $\bar P$ be $\bar B$-modules such
that
\begin{enumerate}
\item There is an exact sequence of $B$-modules
$$
0\to R\to M\stackrel{\varphi}{\longrightarrow} {\bar M}\to 0;
$$
\item $\pi\colon P\to \bar P$ is a homomorphism of $B$-modules;
\item $M$ is $\adic$-adically complete;
\item $P$ and $\bar P$ are finite modules over $B$ and $\bar B$, respectively.
\end{enumerate}
Suppose that we are given natural numbers $r$, $h$, $c$, and $n$
such that
\begin{enumerate}
\item $\adic ^r M\cap \annihil_M(t)=0$ and $\adic ^{ r}{\bar M}\cap \annihil_{\bar M}(t)=0$;
\item $\adic ^{h}\subseteq H_P\subseteq B$ and $\adic ^{ h}\subseteq H_{\bar P}\subseteq \bar B$;
\item $\adic ^k M\cap R = \adic ^{k-c}(\adic ^c M\cap R)$ for all $k\geqslant c$;
\item $n>\max(r+2h,c+2h,4h)$.
\end{enumerate}
Let we be given the following commutative diagram
$$
\xymatrix{
    {M/\adic ^n M}\ar[r]^{\varphi_n}&{{\bar M}/\adic ^n {\bar M}}\ar[r]&{0}\\
    {P}\ar[r]^{\pi}\ar[u]^{\alpha_n}&{\bar P}\ar[u]^{\beta_n}&{}\\
}
$$
where $\varphi_n$ is induced by $\varphi$ (all homomorphisms on the
diagram are the homomorphisms of $B$-modules).

Then, there exist homomorphisms $\alpha\colon P\to M$ and
$\beta\colon \bar P\to {\bar M}$ (homomorphisms of $B$-modules) such
that the following diagram is commutative
$$
\xymatrix{
    {P}\ar[dd]_{\alpha_n}\ar[rd]^{\pi}\ar@{-->}[rr]^{\alpha}&{}&{M}\ar[rd]^{\varphi}\ar[dd]|!{[d];[dr]}\hole&{}\\
    {}&{\bar P}\ar[dd]_(.30){\beta_n}\ar@{-->}[rr]^(.30){\beta}&{}&{{\bar M}}\ar[dd]\\
    {M/\adic ^n M}\ar[dr]^{\varphi_n}\ar[rr]|!{[r];[rd]}\hole&{}&{M/\adic ^{n-2h}M}\ar[rd]^{\varphi_{n-2h}}&{}\\
    {}&{{\bar M}/\adic ^n {\bar M}}\ar[rr]&{}&{{\bar M}/\adic ^{n-2h}{\bar M}}\\
}
$$
\end{lemma}
\begin{proof}
By Lemma~\ref{lemma_step_lift}, one finds homomorphisms
$$
\gamma\colon P\to M/\adic ^{2n-h}M \;\;\mbox{ and } \;\;\delta\colon
\bar P\to M/\adic ^{2n-h}{\bar M}
$$
such that the following diagrams are commutative
$$
\xymatrix{
    {P}\ar@{-->}[r]^-{\gamma}\ar[d]_{\alpha_n}\ar[rd]^{\alpha_{n-h}}&{M/\adic ^{2n-h}M}\ar[d]\\
    {M/\adic ^n M}\ar[r]&{M/\adic ^{n-h}M}\\
}\:\:\: \xymatrix{
    {\bar P}\ar@{-->}[r]^-{\delta}\ar[d]_{\beta_n}\ar[rd]^{\beta_{n-h}}&{{\bar M}/\adic ^{2n-h}{\bar M}}\ar[d]\\
    {{\bar M}/\adic ^n {\bar M}}\ar[r]&{{\bar M}/\adic ^{n-h}{\bar M}}\\
}
$$
Where, $\alpha_{n-h}$ and $\beta_{n-h}$ are the compositions of
$\alpha_n$ and $\beta_n$ with the corresponding quotient map,
respectively.

We will draw everything on the following diagram
$$
\xymatrix@!0@R=40pt@C=50pt{
    {\adic ^{n-2h}M/\adic ^{2(n-2h)}M}\ar[rrr]\ar[rrd]^{i}&&&{\adic ^{n-2h}{\bar M}/\adic ^{2(n-2h)}{\bar M}}\ar[rrd]&&&&\\
    &&{M/\adic ^{2(n-2h)}M}\ar[rrr]^{\varphi_{2(n-2h)}}\ar[rrd]&&&{{\bar M}/\adic ^{2(n-2h)}{\bar M}}\ar[rrd]&&\\
    &&&&{M/\adic ^{n-2h}M}\ar[rrr]^{\varphi_{n-2h}}&&&{{\bar M}/\adic ^{n-2h}{\bar M}}\\
    {\adic ^{n-h} M/\adic ^{2n-h}M}\ar[rrr]|!{[rr];[rrd]}\hole\ar[uuu]\ar[rrd]&&&{\adic ^{n-h}{\bar M}/\adic ^{2n-h}{\bar M}}\ar[rrd]_(0.30){j}|!{[r];[rd]}\hole\ar[uuu]|!{[uul];[uuulll]}\hole|!{[uu];[uur]}\hole&&&&\\
    &&{M/\adic ^{2n-h}M}\ar[rrr]^{\varphi_{2n-h}}|!{[rr];[rrd]}\hole\ar[uuu]^{\nu}\ar[rrd]&&&{{\bar M}/\adic ^{2n-h}{\bar M}}\ar[rrd]\ar[uuu]^{\bar \nu}|!{[uu];[uur]}\hole&&\\
    &&&&{M/\adic ^{n-h} M}\ar[rrr]^{\varphi_{n-h}}|!{[r];[rd]}\hole\ar[uuu]&&&{{\bar M}/\adic ^{n-h} {\bar M}}\ar[uuu]\\
    &&{P}\ar[rrr]^{\pi}\ar@{-->}[uu]^(.65){\gamma}\ar[rru]^{\alpha_{n-h}}\ar@{..>}[ruuu]_<<<<<<<<<<{\xi}\ar@{..>}[lluuuuuu]^<<<<<<<<<<{\eta}&&&{\bar P}\ar@{-->}[uu]^(.65){\delta}\ar[rru]^{\beta_{n-h}}&&\\
}
$$
On this diagram $i$, $j$, $\nu$, $\bar \nu$ are the natural maps. We
will produce all dotted arrows step by step.

Note, that the diagram is commutative only for solid arrows. By
construction, the homomorphism
$$
\varphi_{2n-h}\circ\gamma - \delta\circ\pi
$$
factors through $j$, that is, equals $j\circ \xi$ for some
$$
\xi\colon P\to \adic ^{n-h}{\bar M}/\adic ^{2n-h}{\bar M}.
$$
Then, by Lemma~\ref{lemma_duo_lift}, we find the homomorphism
$$
\eta\colon P\to \adic ^{n-2h}M/\adic ^{2(n-2h)}M.
$$
Thus, homomorphisms
$$
\nu\circ\gamma-i\circ\eta\colon P\to M/\adic ^{2(n-2h)}M
$$
and
$$
\bar\nu\circ\delta\colon \bar P\to {\bar M}/\adic ^{2(n-2h)}{\bar M}
$$
give the required lifting. Now, since $M$ and, hence, ${\bar M}$ are
$\adic $-adically complete, the result follows by induction on $n$.

\end{proof}

\begin{theorem}\label{theorem:duo_mod_algeb}

Let $A$ be a ring, $\adic=(t)\subseteq A$ a principal ideal, we are
given a surjection $B\to \bar B$ of $A$-algebras, $Q$ is a
$\widehat{B}$-module, and $\bar Q$ is a $\widehat{\bar B}$-modules
such that
\begin{enumerate}
\item The pairs $(B, \adic B)$ and $(\bar B, \adic\bar B)$ are
henselian;
\item $\phi\colon Q\to \bar Q$ is a surjective homomorphism of
$\widehat{B}$-modules;
\item $Q_{t}$ is $\widehat{B}_t$-projective;
\item $\bar Q_{t}$ is $\widehat{\bar B}_t$-projective;
\item $Q$ is a finite $\widehat{B}$-module.
\end{enumerate}

Then, there exist a $B$-module $P_0$, a $\bar B$-module $\bar P_0$,
a surjection $\phi_0\colon P_0\to\bar P_0$, and isomorphisms
$\varphi\colon Q\to \widehat{ P}_0$ and $\psi\colon \bar Q\to
\widehat{\bar P}_0$ such that the following diagram is commutative
$$
\xymatrix{
    {Q}\ar[r]^{\phi}\ar[d]^{\varphi}&{\bar Q}\ar[r]\ar[d]^{\psi}&{0}\\
    {\widehat{P_0}}\ar[r]^{\widehat{\phi}_0}&{\widehat{\bar P_0}}\ar[r]&{0}\\
}
$$
\end{theorem}
\begin{proof}
By Theorem~3 of~\cite{Elkik73}, we can algebraize $Q$ and $\bar Q$
separately. So, we may suppose that $Q=\widehat{P}$, $\bar Q =
\widehat{\bar P}$, where $P$ is a finite $B$-module, $\bar P$ is a
finite $\bar B$-module, and, for some natural $h$, we have
$$
\adic ^h\subseteq H_P\subseteq B\;\; \mbox{ and }\;\; \adic
^h\subseteq H_{\bar P}\subseteq \bar B.
$$

We take some free resolutions of $P$ and $\bar P$ as follows
$$
\xymatrix{
    {B^p}\ar[r]^{L}&{B^q}\ar[r]&{P}\ar[r]&{0}\\
    {\bar B^{\bar p}}\ar[r]^{\bar L}&{\bar B^{\bar q}}\ar[r]&{\bar P}\ar[r]&{0}\\
}
$$
Then the homomorphism $\phi$ induces the homomorphisms $u$ and $v$
on the following diagram
$$
\xymatrix{
    {0}&{0}&{}\\
    {\widehat{P}}\ar[u]\ar[r]^{\phi}&{\widehat{\bar P}}\ar[u]\ar[r]&{0}\\
    {\widehat{B}^q}\ar[u]\ar[r]^{u}&{\widehat{\bar B}^{\bar q}}\ar[u]\ar[r]&{0}\\
    {\widehat{B}^p}\ar[u]^{L}\ar[r]^{v}&{\widehat{\bar B}^{\bar p}}\ar[u]^{\bar L}&{}\\
}
$$
where we may suppose that $u$ is surjective by enlarging the $q$.
Here, $L$ is a $q\times p$-matrix over $B$ and $\bar L$ is $\bar
q\times \bar p$-matrix over $\bar B$.

We define the following algebra
$$
D = \bar B[X,Y]/(XL-\bar LY),
$$
where $X$ is a $\bar q\times q$-matrix of indeterminates $x_{ij}$,
and $Y$ is a $\bar p\times p$-matrix of indeterminates $y_{ij}$.
This algebra describes all pairs of $B$-homomorphisms $u_0$, $v_0$
such that the following diagram is commutative
$$
\xymatrix{
    {B^q}\ar[r]^{u_0}&{\bar B^{\bar q}}\\
    {B^p}\ar[r]^{v_0}\ar[u]^{L}&{\bar B^{\bar p}}\ar[u]^{\bar L}\\
}
$$
The algebra $\bar D=\widehat{\bar B}\otimes_{\bar B} D$ describes
the same pairs of $\widehat{B}$-homomorphisms for free modules over
$\widehat{B}$ and $\widehat{\bar B}$.

The pair $u$ and $v$ defines a $\widehat{\bar B}$-section $\bar D\to
\widehat{\bar B}$. So, we have
$$
\xymatrix{
    {\widehat{\bar B}}\ar[r]&{\bar D}\ar@/_6pt/[l]_{(u,v)}\\
    {\bar B}\ar[r]\ar[u]&{D}\ar[u]\ar@{-->}@/_6pt/[l]_{(u_0,v_0)}\\
}
$$
The algebra $D_t$ is smooth over $\bar B_t$ by
Lemma~\ref{lemma:alg_D_poly} below. By Theorem~2~bis
of~\cite{Elkik73}, there is a section $D\to \bar B$ given by a pair
$(u_0,v_0)$ such that the following diagrams are commutative
$$
\xymatrix@!0@R=25pt@C=50pt{
    {}&{\widehat{\bar B}^{\bar q}}\ar[rd]&{}\\
    {\widehat{B}^q}\ar[rd]_{u}\ar[ru]^{\widehat{u_0}}&{}&{\left(\widehat{\bar B}/\adic^n \widehat{\bar B}\right)^{\bar q}}\\
    {}&{\widehat{\bar B}^{\bar q}}\ar[ru]&{}\\
}
    \:\:\:\:
\xymatrix@!0@R=25pt@C=50pt{
    {}&{\widehat{\bar B}^{\bar p}}\ar[rd]&{}\\
    {\widehat{B}^p}\ar[rd]_{v}\ar[ru]^{\widehat{v_0}}&{}&{\left(\widehat{\bar B}/\adic^n \widehat{\bar B}\right)^{\bar p}}\\
    {}&{\widehat{\bar B}^{\bar p}}\ar[ru]&{}\\
}
$$
for sufficiently large $n$. In particular,
$$
\Im \widehat{u}_0 + \rad\widehat{\bar B}^{\bar q} = \Im u =
\widehat{\bar B}^{\bar q}
$$
By Nakayama's lemma, $\widehat{u}_0$ is also surjective. Since $B$
is henselian, the completion $\widehat{B}$ is faithfully flat.
Therefore, $u_0$ is also surjective.

In other words, the pair $(u_0,v_0)$ induces a surjective
$B$-homomorphism $P\stackrel{\phi_0}{\longrightarrow}\bar P\to 0$
and, by definition, we have the following commutative diagram
$$
\xymatrix{
    {\widehat{P}}\ar[r]^{\widehat{\phi}_0}\ar[d]^{\varphi_n}&{\widehat{\bar P}}\ar[r]\ar[d]^{\psi_n}&{0}\\
    {\widehat{P}/\adic ^n\widehat{P}}\ar[r]^{\phi_n}&{\widehat{\bar P}/\adic ^n\widehat{\bar P}}\ar[r]&{0}\\
}
$$
where $\phi_n$ is induced by $\phi$, $\varphi_n$ and $\psi_n$ are
the quotient maps.

Now, by Lemma~\ref{lemma:duo_lift}, we can find homomorphisms
$\varphi\colon \widehat{P}\to \widehat{P}$ and $\psi\colon
\widehat{\bar P}\to \widehat{\bar P}$ such that the following
diagram is commutative
$$
\xymatrix{
    {\widehat{P}}\ar[r]^{\phi_0}\ar[d]^{\varphi}&{\widehat{\bar P}}\ar[r]\ar[d]^{\psi}&{0}\\
    {\widehat{P}}\ar[r]^{\phi}&{\widehat{\bar P}}\ar[r]&{0}\\
}
$$
and, by construction, $\varphi$ and $\psi$ equal $\varphi_n$ and
$\psi_n$ modulo $\adic ^{n-2h}$, respectively. Thus they are equal
to the identity maps modulo $\adic ^{n-2h}$ and therefore are also
isomorphisms.
\end{proof}

\begin{lemma}\label{lemma:alg_D_poly}
Let the conditions of Theorem~\ref{theorem:duo_mod_algeb} hold and
the $\bar B$-algebra $D$ is defined as in the proof of the
proposition. Then, for every prime ideal $\mathfrak p\subseteq \bar
B$ not containing $t$, $D_{\mathfrak p}$ is a polynomial ring over
$\bar B_{\mathfrak p}$.
\end{lemma}
\begin{proof}

The matrix $L$ also defines the $\bar B$-module $P'=P\otimes_{B}\bar
B$. This module is projective over the complement of $V(\adic)$.
Therefore, we have the following split exact sequence of free $\bar
B_{\mathfrak p}$-modules
$$
0\to K_{\mathfrak p}\to \bar B^{q}_{\mathfrak p}\to P'_{\mathfrak
p}\to 0
$$
where $K$ is the module generated by the columns of $L$. Since
$K_{\mathfrak p}$ is free of some rank $d$ and $\bar B_{\mathfrak
p}$ is local, one can find matrices $U\in \mathrm{GL}_{q}(\bar
B_\mathfrak p)$ and $V\in \mathrm{GL}_p(\bar B_{\mathfrak p})$ such
that the matrix $ULV$ is of the form
$$
\begin{pmatrix}
    E&0\\
    0&0
\end{pmatrix}
$$
where $E$ is the identity matrix of size $d$. By the same arguments,
there exist matrices $\bar U\in \mathrm{GL}_{\bar q}(\bar
B_{\mathfrak p})$ and $\bar V\in \mathrm{GL}_{\bar p}(\bar
B_{\mathfrak p})$ such that the matrix $\bar U\bar L\bar V$ is of
the form
$$
\begin{pmatrix}
    E&0\\
    0&0
\end{pmatrix}
$$
where $E$ is the identity matrix of some size $\bar d$. Now, we
change the generators $X$ and $Y$ by $X'=\bar U X U^{-1}$ and
$Y'={\bar V}^{-1}Y V$, respectively. Then, the algebra $D_{\mathfrak
p}$ is given by
$$
D_{\mathfrak p} = \bar B_{\mathfrak p} [X',Y']/(F)
$$
where $F$ is a $\bar q\times p$-matrix of the form
$$
\begin{pmatrix}
    X'_{d}&0\\
\end{pmatrix}-
\begin{pmatrix}
    Y'_{\bar d}\\
    0
\end{pmatrix}
$$
Here $X'_{d}$ is the matrix consisting of the first $d$ columns of
the matrix $X'$ and $Y'_{\bar d}$ is the matrix consisting of the
first $\bar d$ rows of the matrix $Y'$. Now, the claim is clear.
\end{proof}

\begin{corollary}\label{cor:duo_mod_algeb}

Let $A$ be a ring, $\adic=(t)\subseteq A$ a principal ideal,
$f\colon B\to \bar B$ is a surjection of $A$-algebras, $Q$ is a
$\widehat{B}$-module, and $\bar Q$ is a $\widehat{\bar B}$-modules
such that
\begin{enumerate}
\item $\phi\colon Q\to \bar Q$ is a surjective homomorphism of
$\widehat{B}$-modules;
\item $Q_{t}$ is $\widehat{B}_t$-projective;
\item $\bar Q_{t}$ is $\widehat{\bar B}_t$-projective;
\item $Q$ is a finite $\widehat{B}$-module.
\end{enumerate}

Then, there exists a surjective homomorphism of $A$-algebras
$f'\colon B'\to \bar B'$ such that the following diagram is
commutative
$$
\xymatrix{
    {B'}\ar[r]^{f'}&{\bar B'}\ar[r]&{0}\\
    {B}\ar[r]^{f}\ar[u]&{\bar B}\ar[r]\ar[u]&{0}\\
}
$$
where $B'$ and $\bar B'$ are \'etale extensions of $B$ and $\bar B$,
respectively. There also exist a $B'$-module $P_0$, a $\bar
B'$-module $\bar P_0$, a surjection
$P_0\stackrel{\phi_0}{\longrightarrow}\bar P_0\to 0$, and
isomorphisms $\varphi\colon Q\to \widehat{ P}_0$ and $\psi\colon
\bar Q\to \widehat{\bar P}_0$ such that the following diagram is
commutative
$$
\xymatrix{
    {Q}\ar[r]^{\phi}\ar[d]^{\varphi}&{\bar Q}\ar[r]\ar[d]^{\psi}&{0}\\
    {\widehat{P_0}}\ar[r]^{\widehat{\phi}_0}&{\widehat{\bar P_0}}\ar[r]&{0}\\
}
$$
\end{corollary}

\subsection{Pairs of algebras}

\subsubsection{A particular case}

\begin{proposition}\label{prop_good_case}
Let $A$ be a ring $\adic=(t)\subseteq A$ is a principal ideal, and
we are given the following commutative diagram
$$
\xymatrix@R=15pt{
    {0}\ar[r]&{J}\ar[r]\ar@{^(->}[d]&{\widehat{A}\{X\}}\ar[r]\ar@{=}[d]&{B}\ar[r]\ar@{->>}[d]^{\pi}&{0}\\
    {0}\ar[r]&{\bar J}\ar[r]&{\widehat{A}\{X\}}\ar[r]&{\bar B}\ar[r]&{0}\\
}
$$
where $B$ and $\bar B$ are formally finitely generated
$\widehat{A}$-algebras, $X=\{x_1,\ldots,x_n\}$, $J$ and $\bar J$ the
corresponding ideals, and $\pi$ is surjective. Assume that $B$ and
$\bar B$ are formally smooth over the complement of
$V(\widehat{\adic})$ and additionally we have
\begin{enumerate}
\item $(J/J^2)_t$ is a free $B_t$-module of rank $d$;
\item $(\bar J/\bar J^2)_t$ is a free $\bar B_t$-module of rank $\bar
d$;
\item $(\bar J/\bar J^2)_t = J/J^2\otimes_{B}\bar B_t \oplus D$,
where $D$ is a free $\bar B_t$-module of rank $k = \bar d - d$.
\end{enumerate}
Then there exist a surjective homomorphism of $A$-algebras
$\pi_0\colon D\to \bar D$, and isomorphisms $\varphi\colon B\to
\widehat{D}$ and $\psi\colon \bar B\to \widehat{\bar D}$ such that
$D$ and $\bar D$ are smooth over the complement of $V(\adic)$ and
the following diagram is commutative
$$
\xymatrix{
    &{B}\ar[rr]^{\pi}\ar[dl]_{\varphi}&&{\bar B}\ar[r]\ar[dl]_{\psi}&0\\
    {\widehat{D}}\ar[rr]^{\widehat{\pi_0}}&&{\widehat{\bar D}}\ar[r]&0&\\
    {D}\ar[rr]^{\pi_0}\ar[u]&&{\bar D}\ar[r]\ar[u]&0&\\
}
$$
\end{proposition}
\begin{proof}

Firstly, we explicitly choose some ``good'' generators of ideals $J$
and $\bar J$. We take the generators
\begin{align*}
    J &= (f_1,\ldots, f_d,f_{d+1},\ldots, f_m)\\
    \bar J &= (f_1,\ldots, f_m, g_1,\ldots,g_k,g_{k+1},\ldots,g_{s})
\end{align*}
such that the images of $f_1,\ldots, f_d$ form the basis of
$(J/J^2)_t$ and the images of $g_1,\ldots, g_k$ form the basis of
$D$. So, we have
\begin{align*}
    (J/J^2)_t &= (f_1,\ldots, f_d)_t\\
    (\bar J/\bar J^2)_t &=(f_1,\ldots,f_d,g_1,\ldots, g_k)_t
\end{align*}
Therefore, there exist elements $e\in J$ and $q\in \bar J$ and a
natural $h$ such that
\begin{align*}
    J_t &= (f_1,\ldots, f_d, e)_t\\
    \bar J_t &= (f_1,\ldots, f_d,g_1,\ldots, g_k, q)_t
\end{align*}
and
\begin{align*}
    e^2 &= t^h e \pmod{f_1,\ldots,f_d} & \mbox{ in } &\widehat{A}\{X\}\\
    q^2 &= t^h q \pmod{f_1,\ldots,f_d,g_1,\ldots,g_k} & \mbox{ in } &\widehat{A}\{X\}\\
\end{align*}

The condition~(1) implies that there is a natural $h'$ such that
$$
t^{h'}\in \triangle^d(f_1,\ldots,f_d).
$$
Enlarging $h$ or $h'$, we can assume that $h=h'$. The conditions~(2)
and~(3) imply that there exists a natural $h'$ such that
$$
t^{h'}\in\triangle^{d+k}(f_1,\ldots,f_d,g_1,\ldots, g_k).
$$
Again, we can assume that $h=h'$.

Now, we consider the following relations on $f_i,g_i,e,q$:
\begin{align*}
    \Phi_1 &= e^2 - t^h e + a_1 f_1 + \ldots + a_d f_d\\
    \Phi_2 &= q^2 - t^h q + b_1 f_1 + \ldots + b_d f_d +
                            c_1 g_1 + \ldots + c_k g_k\\
    K_j &= t^h f_{d+j} + u_{j1} f_1 + \ldots + u_{jd} f_d + u_i e\\
    L_j &= t^h g_{k+j} + w_{j1} f_1 + \ldots + w_{jd} f_d +
                         v_{j1} g_1 + \ldots + v_{jk} g_k + v_i q
\end{align*}

We take $A[X]^h$, that is, the henselization of $A[X]$ with respect
to ideal $\adic[X]$, and denote by $\adic[X]^h$ the corresponding
ideal in $A[X]^h$. We define the following algebra
$$
D =
A[X]^h[F_i,G_i,E,Q,A_i,B_i,C_i,U_{ij},V_{ij},W_{ij},U_i,V_i]/(\Phi_1,\Phi_2,K_j,L_j)
$$
where we use capital letters to denote the corresponding
indeterminates. The system of equations $\Phi_1,\Phi_2,K_j,L_j$ will
be denoted by $\Sigma$. Also we set
$$
\bar D = \widehat{A}\{X\}\otimes_{A[X]^h} D.
$$

The elements $f_i,g_i,e,q$, etc. give a section $\bar
\varepsilon\colon \bar D\to \widehat{A}\{X\}$. Let us show that
$$
t^{h(m+s-d-k+2)}\in \bar \varepsilon \left(H_{\bar D/
\widehat{A}\{X\}}\right)
$$
For we show that
$$
t^{h(m+s-d-k+2)}\in\bar \varepsilon
\left(\triangle^{m+s-d-k+2}(\Sigma) \right)
$$
Indeed, the Jacobian matrix of $\Sigma$ at $f_i,g_i,e,q$, etc. is
the following

\begin{tabular}{l|ccccccccccccccc}
                &$f_{d+j}$  &$g_{k+j}$  &$e$        &$a_i$ &$q$        &$b_i$   &$c_i$  &other indeterminates\\
                \hline
    $\Phi_1$    &$0$        &$0$        &$2e-t^h$   &$f_i$  &$0$        &$0$    &$0$    &*\\
    $\Phi_2$    &$0$        &$0$        &$0$        &$0$    &$2q-t^h$   &$f_i$  &$g_i$  &*\\
    $K_1$       &$0$        &$0$        &*          &$0$    &$0$        &$0$    &$0$    &*\\
    $K_j$       &$t^h$      &$0$        &*          &$0$    &$0$        &$0$    &$0$    &*\\
    $K_{m-d}$   &$0$        &$0$        &*          &$0$    &$0$        &$0$    &$0$    &*\\
    $L_1$       &$0$        &$0$        &$0$        &$0$    &*          &$0$    &$0$    &*\\
    $L_{j}$     &$0$        &$t^h$      &$0$        &$0$    &*          &$0$    &$0$    &*\\
    $L_{s-k}$   &$0$        &$0$        &$0$        &$0$    &*          &$0$    &$0$    &*\\
\end{tabular}

Therefore, we have the following inclusion
$$
t^{h(m-d)}(2e-t^h, f_i)t^{h(s-k)}(2q-t^h,f_i,g_j)\subseteq
\bar\varepsilon \left(\triangle^{m+s-d-k+2}(\Sigma)\right),
$$
where $1\leqslant i \leqslant d$ and $1\leqslant j\leqslant k$. Now,
by the choice of $e$, we have
$$
(2e-t^h)(2e-t^h) = t^{2h} \pmod{f_1,\ldots,f_d}\; \mbox{ in } \;
\widehat{A}\{X\}
$$
Therefore, $t^{2h}\in (2e-t^h,f_i)$. Absolutely analogous
calculation shows that $t^{2h}\in (2q-t^h,f_i,g_j)$ and we get what
we need.

By Theorem~2~bis of~\cite{Elkik73}, the section $\bar\varepsilon$
can be approximated by a section $\varepsilon\colon D\to A[X]^h$
modulo $\adic^n$ for any sufficiently large $n$. We denote the
corresponding elements by the same letters with upper index $0$,
that is, $f^0_i, g^0_i, e^0,q^0$, etc. and define the following
ideals and algebras
\begin{align*}
    J^0 &= (f^0_1,\ldots,f^0_m,e^0)\subseteq A[x]^h & \bar J^0 &=
    (f^0_1,\ldots,f^0_m,e^0,g^0_1,\ldots,g^0_s,q^0)\subseteq A[x]^h\\
    B^0 &= A[X]^h/J^0  &   \bar B^0 &= A[X]^h/\bar J^0
\end{align*}
Because of the choice of $f^0_i$ and $e^0$, we have
$$
\triangle^{d}(f^0_1,\ldots,f^0_d) = \triangle^{d}(f_1,\ldots,f_d)
\pmod{\widehat{\adic}^n\{X\}}
$$
If $n>h$, then $t^h$ belongs to the left-hand part. Similarly, we
get that $t^h$ belongs to
$\triangle^{d+k}(f^0_1,\ldots,f^0_d,g^0_1,\ldots,g^0_k)$. Because of
$K_j$ and $L_j$, we see that
$$
J^0_t = (f^0_1,\ldots,f^0_d,e^0)_t \; \mbox{ and }\; \bar J^0_t =
(f^0_1,\ldots,f^0_d,g^0_1,\ldots,g^0_k,q^0)_t
$$
Since $e^0/t^h$ is idempotent modulo $(f^0_i)$ and $q^0/t^h$ is
idempotent modulo $(f^0_i,g^0_j)$, the algebras $B^0$ and $\bar B^0$
are smooth over the complement of $V(\adic)$ by the Jacobian
criterion~\cite[Theorem~22.6.1]{EGA1964_20}.

By definition, we have a natural surjection $\pi^0\colon B^0\to \bar
B^0$, and the following diagram is commutative
$$
\xymatrix{
    {B^0}\ar[r]^{\pi^0}\ar[d]^{\alpha_n}&{\bar B^0}\ar[r]\ar[d]^{\bar\alpha_n}&{0}\\
    {B^0/\adic^n B^0}\ar[r]^{\pi^0_n}\ar@{=}[d]&{\bar B^0/\adic^n \bar B^0}\ar[r]\ar@{=}[d]&{0}\\
    {B/\adic^n B}\ar[r]^{\pi_n}&{\bar B/\adic^n \bar B}\ar[r]&{0}\\
}
$$
where $\pi_n$ is induced by $\pi$, $\pi^0_n$ is induced by $\pi^0$,
and the vertical maps $\alpha_n$ and $\bar \alpha_n$ are the
quotient maps. Then, by Proposition~\ref{prop:isom_lift}, there exists isomorphisms $\alpha\colon B^0\to B$ and $\bar \alpha \colon \bar B^0\to \bar B$ such that the following diagram is commutative
$$
\xymatrix@!0@R=40pt@C=70pt{
    {}&{\widehat{B^0}}\ar[rr]^(0.6){\psi}\ar@{-->}[dl]_{\alpha}\ar[dd]|!{[dl];[d]}\hole&{}&{\widehat{\bar B^0}}\ar[dd]\ar@{-->}[dl]_{\beta}\\
    {B}\ar[rr]^(0.6){\phi}\ar[dd]&{}&{\bar B}\ar[dd]&{}\\
    {}&{B_0/\adic^{n-2h}B_0}\ar[rr]^(0.4){\psi_{n-2h}}|!{[ru];[r]}\hole\ar[dl]_{\alpha_{n-2h}}&{}&{\bar B_0/\adic^{n-2h}\bar B_0}\ar[dl]_{\beta_{n-2h}}\\
    {B/\adic^{n-2h}B}\ar[rr]^(0.4){\phi_{n-2h}}&{}&{\bar B/\adic^{n-2h}\bar B}&{}\\
}
$$

Algebras $B^0$ and $\bar B^0$ are quotients of $A[X]^h$ by some
finitely generated ideals. Also, expressing $t^h$ as an element of
the ideals
$$
\triangle^d(f_1^0,\ldots,f_d^0) \quad \mbox{ and } \quad
\triangle^{d+k}(f_1^0,\ldots,f_d^0,g_1^0,\ldots,g_k^0)
$$
we use finitely many elements of $A[x]^h$. So, we can take a
strictly \'etale extension $A[X]^0$ of $A[X]$ containing all these
elements and solutions of $\Sigma$ in $A[x]^h$. Then the required
algebras are
$$
D=A[X]^0/\left(J^0\cap A[X]^0\right) \; \mbox{ and }\; \bar D=
A[X]^0/\left(\bar J^0\cap A[X]^0\right)
$$
and the map $\pi_0\colon D\to \bar D$ is the natural quotient map.

\end{proof}

\subsubsection{Reduction of the general case}

If we are given a ring $A$ with an ideal $\adic$, $B$ is an
$\adic$-adically complete $A$-algebra, and $M$ is a $B$-module. Then
the symmetric algebra of $M$ over $B$ will be denoted by $S_B(M)$.
The completion of $S_B(M)$ in the topology generated by $\adic $
will be denoted by $S_B\{M\}$.

\begin{lemma}\label{lemma:approx_inverse}
Let $A$ be a ring, $\adic\subseteq A$ an ideal, and we are given two
$A$-algebras $B$ and $C$ together with two homomorphisms
$\Phi,\Psi\colon B\to C$ over $A$ such that
\begin{enumerate}
\item $B$ is $\adic B$-adically complete;
\item $C$ is $\adic C$-adically complete;
\item $\Psi$ is an isomorphism;
\item $\Im\gamma\subseteq \adic C$, where $\gamma = \Phi - \Psi$.
\end{enumerate}
Then $\Phi$ is also an isomorphism.
\end{lemma}
\begin{proof}
To prove the result, we should show that $\Phi = \Psi + \gamma$ has
an inverse map. We define
$$
\varphi = \Psi^{-1}\circ \gamma\colon B\to B
$$
and we have $\Phi = \Psi\circ (Id_B + \varphi)$. It is enough to
show that $Id_B+\varphi$ is an invertible homomorphism of
$A$-modules. Since $\Psi^{-1}$ is also an $A$-homomorphism, $\Im
\varphi\subseteq \adic B$. The ring $B$ is $\adic B$-adically
complete and, for any element $x\in B$, $\varphi(x)\in \adic B$.
Therefore, the element
$$
\xi(x) = \sum_{n=0}^\infty (-1)^n\varphi^n(x)
$$
is well-defined. Since $\Phi$, $\Psi$, and $\Psi^{-1}$ are
$A$-homomorphisms, they are continuous. Therefore, the map $\varphi$
is also continuous. Now, an easy calculation shows that
$\xi=(Id_B+\varphi)^{-1}$.
\end{proof}

\begin{lemma}\label{lemma:auto_smooth}
Let $A$ be a ring and $\adic\subseteq A$ an ideal, and $B$ is a
finitely generated $A$-algebra such that its completion
$\widehat{B}$ is smooth over the complement of $V(\widehat{\adic})$.
Then there exists an $s\in \adic$ such that $B_{1+s}$ is already
smooth over the complement of $V(\adic)$.
\end{lemma}
\begin{proof}

We represent algebra $B$ as the following quotient $A[x]/J$, where
$x=\{x_1,\ldots,x_n\}$ is a finite set of indeterminates. Let us
denote the localization $B_{1+\adic B}$ by $B'$ and $J_{1+\adic B}$
by $J'$. Then we have the following sequence of homomorphisms $B\to
B'\to \widehat{B}$, and the following exact sequence
$$
0\to \widehat{J'}\to \widehat{A}\{x\}\to \widehat{B}\to 0
$$
Thus, by Lemmas~\ref{lemma:HB_computation}
and~\ref{lemma:fHB_computation}, we have
\begin{align*}
  H_{B'/A}&=\annihil_B\left(\hom_{B'}\left(J'/J'^2,J'/J'^2\right)/\hom_{A[x]}\left(\Omega_{A[x]/A},J'/J'^2\right)\right)\\
  \bar H_{\widehat{B}/\widehat{A}}&=\annihil_{\widehat{B}}\left(\hom_{\widehat{B}}\left(\widehat{J'}/\widehat{J'}^2,\widehat{J'}/\widehat{J'}^2\right)/\hom_{\widehat{A}\{x\}}\left(\Omega^s_{\widehat{A}\{x\}/\widehat{A}},\widehat{J'}/\widehat{J'}^2\right)\right)
\end{align*}
Since $\widehat{B}=\widehat{B'}$ is a faithfully flat $B'$-module
and the modules $J'/J'^2$ and $\Omega_{A[x]/A}$ are finitely
generated, we have
$$
\bar H_{\widehat{B}/\widehat{A}} = H_{B'/A} \otimes_{B'}\widehat{B}
= H_{B'/A}\widehat{B}
$$
and
$$
B'\cap \bar H_{\widehat{B}/\widehat{A}} = H_{B'/A}.
$$
Formal smoothness of $\widehat{B}$ over the complement of
$V(\widehat{\adic})$ means that, for some $h$, $\adic^h\subseteq
\bar H_{\widehat{B}/\widehat{A}}$. Therefore, $\adic^h\subseteq
H_{B'/A}$. The ideal $\adic^h$ is finitely generated, hence, there
is an element $s\in \adic$ such that $\adic^h\subseteq
H_{B_{1+s}/A}$.
\end{proof}

\begin{lemma}\label{lemma_reduction}
Let $A$ be a ring, $\adic=(t)\subseteq A$ a principal ideal, and we
are given a surjective homomorphism $B\to \bar B$ of
$\adic$-adically complete $\widehat{A}$-algebras. Let $P$ and $\bar
P$ be finite modules over $B$ and $\bar B$, respectively, such that
\begin{enumerate}
\item There is a surjective homomorphism of $B$-modules $P\to \bar
P$;
\item $P_t$ is $B_t$-projective;
\item $\bar P_t$ is $\bar B_t$-projective.
\end{enumerate}
Suppose that there exist a surjective homomorphism of finitely
generated $A$-algebras $C\to \bar C$ together with isomorphisms
$\varphi' \colon S_B\{P\}\to \widehat C$ and $\psi'\colon S_{\bar
B}\{\bar P\}\to \bar C$ such that $C_t$ and $\bar C_t$ are
$A_t$-smooth and the following diagram is commutative
$$
\xymatrix{
    &{S_{B}\{P\}}\ar[rr]\ar[dl]_{\varphi'}&&{S_{\bar B}\{\bar P\}}\ar[r]\ar[dl]_{\psi'}&0\\
    {\widehat{C}}\ar[rr]&&{\widehat{\bar C}}\ar[r]&0&\\
    {C}\ar[rr]\ar[u]&&{\bar C}\ar[r]\ar[u]&0&\\
}
$$
Then there exist a surjective homomorphism of finitely generated
$A$-algebras $D\to \bar D$ together with isomorphisms $\varphi
\colon B\to \widehat D$ and $\psi\colon \bar B\to \widehat{\bar D}$
such that $D_t$ and $\bar D_t$ are $A_t$-smooth and the following
diagram is commutative
$$
\xymatrix{
    &{B}\ar[rr]\ar[dl]_{\varphi}&&{\bar B}\ar[r]\ar[dl]_{\psi}&0\\
    {\widehat{D}}\ar[rr]&&{\widehat{\bar D}}\ar[r]&0&\\
    {D}\ar[rr]\ar[u]&&{\bar D}\ar[r]\ar[u]&0&\\
}
$$
If additionally, $B$ and $\bar B$ are formally smooth over the
complement of $V(\widehat{\adic})$, the algebras $D_t$ and $\bar
D_t$ are $A_t$-smooth.
\end{lemma}
\begin{proof}
Let us note that, by Lemma~\ref{lemma:auto_smooth}, it is enough to
show the existence of such algebras $D$ and $\bar D$ without showing
the smoothness condition. Now, we denote $S_{B}\{P\}$ and $S_{\bar
B}(\bar P)$ by $R$ and $\bar R$, respectively, $R\otimes_{B}P$ and
$\bar R\otimes_{\bar B}\bar P$ by $M$ and $\bar M$, respectively.

From Corollary~\ref{cor:duo_mod_algeb}, it follows that we can
replace $C$ and $\bar C$ by some \'etale extensions such that, for
some finite $C$-module $M_0$, some finite $\bar C$-module $\bar
M_0$, and a surjective homomorphism $\phi_0\colon M_0\to \bar M_0$,
there are isomorphisms $\varphi\colon M\to \widehat{M}_0$ and
$\psi\colon \bar M\to \widehat{\bar M}_0$ such that the following
diagram is commutative
$$
\xymatrix{
    {M}\ar[r]\ar[d]^{\varphi}&{\bar M}\ar[r]\ar[d]^{\psi}&{0}\\
    {\widehat{M}_0}\ar[r]^{\widehat{\phi}_0}&{\widehat{\bar M}_0}\ar[r]&{0}\\
}
$$

The homomorphism $\phi_0$ induces the following surjective
homomorphism of rings
$$
S_{C}(M_0)\to S_{\bar C}(\bar M_0)\to 0
$$
We can write down the following sequence of isomorphisms for their
completions
$$
\xymatrix{
    {\widehat{S_{C}(M_0)}}\ar@/_35pt/[ddd]_{\nu}\ar[r]\ar[d]&{\widehat{S_{\bar C}(\bar M_0)}}\ar@/^35pt/[ddd]^{\bar \nu}\ar[r]\ar[d]&{0}\\
    {S_{\widehat{C}}\{\widehat{M}_0\}}\ar[r]\ar[d]_{S_{\widehat{C}}\{\varphi\}}&{S_{\widehat{\bar C}}\{\widehat{\bar M}_0\}}\ar[r]|-(0.40)\hole \ar[d]^{S_{ \widehat{\bar C} }\{ \psi \}}&{0}\\
    {S_{R}\{M\}}\ar[r]\ar[d]&{S_{\bar R}\{\bar M\}}\ar[r]|-(0.40)\hole\ar[d]&{0}\\
    {\widehat{R\otimes_{B}R}}\ar[r]&{\widehat{\bar R\otimes_{\bar B}\bar R}}\ar[r]&{0}\\
}
$$
where the vertical arrows are isomorphisms over $\widehat{A}$. Then
the product maps
$$
\mathrm{pr}\colon R\otimes_B R\to B \quad \mbox{and}
\quad\mathrm{pr}\colon \bar R\otimes_{\bar B}\bar R\to \bar R
$$
induce the corresponding homomorphisms on the completions and, thus,
give the following commutative diagram
$$
\xymatrix{
    {\widehat{R\otimes_BR}}\ar[r]\ar[d]^{\mathrm{pr}}&{\widehat{\bar R\otimes_{\bar B}\bar R}}\ar[r]\ar[d]^{\mathrm{pr}}&{0}\\
    {R}\ar[r]&{\bar R}\ar[r]&{0}\\
}
$$
So, we have the following commutative cube
$$
\xymatrix{
    &{S_{R}\{\widehat{M}_0\}}\ar[rr]\ar[ld]_{\Psi}&&{S_{\bar R}\{\widehat{\bar M}_0\}}\ar[rr]\ar[ld]_{\bar \Psi}&&{0}\\
    {R}\ar[rr]&&{\bar R}\ar[rr]&&{0}&\\
    &{S_{C}(M_0)}\ar[rr]|!{[r];[ru]}\hole\ar[uu]|!{[u];[ur]}\hole &&{S_{\bar C}(\bar M_0)}\ar[rr]\ar[uu]|!{[u];[ur]}\hole&&{0}\\
    {C}\ar[rr]\ar[uu]&&{\bar C}\ar[rr]\ar[uu]&&{0}&\\
}
$$
where $\Psi = \mathrm{pr}\circ \nu$ and $\bar \Psi =
\mathrm{pr}\circ \bar \nu$. By Corollary~\ref{cor:duo_alg_lift}, for
any natural number $n$, we can replace $C$ and $\bar C$ by some
\'etale extensions such that there exist homomorphisms $\Phi\colon
S_{C}(M_0)\to C$ and $\bar \Phi\colon S_{\bar C}(\bar M_0)\to \bar
C$ such that $\Phi$ and $\bar \Phi$ coincide with $\Psi$ and
$\bar\Psi$ modulo $\adic ^n$ and the square
$$
\xymatrix{
    {S_{C}(M_0)}\ar[r]\ar[d]^{\Phi}&{S_{\bar C}(\bar M_0)}\ar[r]\ar[d]^{\bar \Phi}&{0}\\
    {C}\ar[r]&{\bar C}\ar[r]&{0}\\
}
$$
is commutative. For our purpose, it is enough to take $n=1$. By
construction, the restrictions of $\Psi$ and $\bar \Psi$ to
$S_{B}\{\widehat{M}_0\}$ and $S_{\bar B}\{\widehat{\bar M}_0\}$,
respectively, induce isomorphisms onto $R$ and $\bar R$,
respectively. Since $\Phi$ and $\bar \Phi$ coincide with $\Psi$ and
$\bar \Psi$ modulo $\adic ^n$, by Lemma~\ref{lemma:approx_inverse},
the restrictions of $\widehat{\Phi}$ and $\widehat{\bar \Phi}$ to
the same subrings are also isomorphisms. Denoting this restrictions
by the same names, we get the following commutative diagram
$$
\xymatrix{
    &{S_{B}\{\widehat{M}_0\}}\ar[rr]\ar[dd]_(0.30){\pi}|!{[d];[dr]}\hole\ar[ld]_{\widehat{\Phi}}&&{S_{\bar B}\{\widehat{\bar M}_0\}}\ar[rr]\ar[dd]_(0.30){\bar \pi}|!{[d];[dr]}\hole\ar[ld]_{\widehat{\bar \Phi}}&&{0}\\
    {R}\ar[rr]\ar[dd]_(0.30){p}&&{\bar R}\ar[rr]\ar[dd]_(0.30){\bar p}&&{0}&\\
    &{B}\ar[rr]|!{[r];[rd]}\hole\ar[ld]_{\Phi'}&&{\bar B}\ar[rr]\ar[ld]_{\bar\Phi'}&&{0}\\
    {\widehat{C/\Phi(M_0)}}\ar[rr]&&{\widehat{\bar C/\bar \Phi(\bar M_0)}}\ar[rr]&&{0}&\\
}
$$
where $\pi$ and $\bar\pi$ are the quotient maps by the ideals
generated by $\widehat{M}_0$ and $\widehat{\bar M}_0$, respectively,
$p$ and $\bar p$ are the quotient maps by the ideals generated by
$\widehat{\Phi}(\widehat{M}_0)$ and $\widehat{\bar
\Phi}(\widehat{\bar M}_0)$, respectively, and $\Phi'$ and $\bar
\Phi'$ are the isomorphisms induced by $\widehat{\Phi}$ and
$\widehat{\bar \Phi}$, respectively. We set
$$
D = C/ \Phi(M_0) \quad \mbox{ and } \quad \bar D = \bar C/ \bar
\Phi(\bar M_0).
$$

Now, if $B$ and $\bar B$ are formally smooth over the complement of
$V(\widehat{\adic})$, then by Lemma~\ref{lemma:auto_smooth}, we can
suppose that their algebraizations $D$ and $\bar D$ are also smooth
over the complement of $V(\adic)$.

\end{proof}

\begin{theorem}\label{theorem:main}
Let $A$ be a ring, $\adic=\left(t\right)\subseteq A$ a principal
idea, and we are given a surjective homomorphism $B\to \bar B$ of
formally finitely generated $\widehat{A}$-algebras such that $B$ and
$\bar B$ are formally smooth over the complement of
$V\left(\widehat{\adic}\right)$. Then there exist a surjective
homomorphism $D\to \bar D$ of finitely generated $A$-algebras being
smooth over the complement of $V\left(\adic\right)$ and two
isomorphisms $\varphi\colon B\to \widehat D$ and $\psi\colon \bar
B\to \widehat{\bar D}$ such that the following diagram is
commutative
$$
\xymatrix{
    &{B}\ar[rr]\ar[dl]_{\varphi}&&{\bar B}\ar[r]\ar[dl]_{\psi}&0\\
    {\widehat{D}}\ar[rr]&&{\widehat{\bar D}}\ar[r]&0&\\
    {D}\ar[rr]\ar[u]&&{\bar D}\ar[r]\ar[u]&0&\\
}
$$

\end{theorem}
\begin{proof}

Using Lemma~\ref{lemma_reduction}, we will reduce this theorem to
Proposition~\ref{prop_good_case}.

Firstly, we present algebras $B$ and $\bar B$ as quotients of a ring
of the form $\widehat{A}\{x\}$, where $x=\{x_1,\ldots,x_n\}$ is a
finite set of indeterminates, as follows
$$
\xymatrix@R=10pt@C=10pt{
    0\ar[r]&{J}\ar[r]\ar[d]&{\widehat{A}\{x\}}\ar[r]\ar@{=}[d]&{B}\ar[r]\ar[d]&0\\
    0\ar[r]&{{\bar J}}\ar[r]&{\widehat{A}\{x\}}\ar[r]&{\bar B}\ar[r] & 0\\
}
$$
Considering ${\bar J}/{\bar J}^2$ as a $B$ module, we find an
epimorphism $B^m\to {\bar J}/{\bar J}^2$. By
Lemma~\ref{lemma_reduction}, we can replace the homomorphism
$B\to\bar B\to 0$ by the homomorphism $S_{B}\{B^m\}\to S_{\bar
B}\{{\bar J}/{\bar J}^2\}\to 0$. So, we may assume that
$\Omega^s_{\bar B/\widehat{A}}$ is free over the complement of
$V\left(\widehat{\adic} \right)$ of some rank $\bar d$.

Again by Lemma~\ref{lemma_reduction}, we can replace $B\to\bar B\to
0$ by the composition $S_{B}\{J/J^2\}\to B\to\bar B$ and assume that
$\Omega^s_{B/\widehat{A}}$ is free over the complement of
$V\left(\widehat{\adic} \right)$ of some rank $d$.

If we present our algebras in the form
$$
B=B\{t_1,\ldots,t_s\}/\left(t_1,\ldots,t_s\right)\quad \mbox{ and
}\quad\bar B=\bar B\{t_1,\ldots,t_s\}/\left(t_1,\ldots,t_s\right),
$$
where $s\geqslant \max\left(d,\bar d\right)$. Then, repeating the
proof of Lemma~3 of~\cite{Elkik73} with $\Omega^s$ instead of
$\Omega$, we can assume that the modules $(J/J^2)_t$ and $({\bar
J}/{\bar J}^2)_t$ are free over $B_t$ and $\bar B_t$, respectively.

Now, we write down the second fundamental sequence for both algebras
with respect to the chosen representations and restrict it to the
complement of $V\left(\widehat{\adic} \right)$.
$$
\xymatrix{
    0\ar[r]&{\left(J/J^2\right)_t}\ar[d]\ar[r]&{\Omega^s_{\widehat{A}\{z\}/\widehat{A}}\otimes_{\widehat{A}\{z\}} B_t}\ar[d]\ar[r]&{\left(\Omega^s_{B/\widehat{A}}\right)_t}\ar[d]\ar[r]&0\\
    0\ar[r]&{\left(J/J^2\right)\otimes_B\bar B_t}\ar@{^(->}[d]\ar[r]&{\Omega^s_{\widehat{A}\{z\}/\widehat{A}}\otimes_{\widehat{A}\{z\}} \bar B_t}\ar@{=}[d]\ar[r]&{\Omega^s_{B/\widehat{A}}\otimes_B \bar B_t}\ar[d]\ar[r]&0\\
    0\ar[r]&{\left({\bar J}/{\bar J}^2\right)_t}\ar[r]&{\Omega^s_{\widehat{A}\{z\}/\widehat{A}}\otimes_{\widehat{A}\{z\}} \bar B_t}\ar[r]&{\left(\Omega^s_{\bar B/\widehat{A}}\right)_t}\ar[r]&0\\
}
$$
Here $z=\{z_1,\ldots,z_n\}$ is the set of indeterminates such that
$J/J^2$ and $\bar J/{\bar J}^2$ are free over the complement of
$V\left(\widehat{\adic}\right)$. All these sequences are split
exact. In particular,
$$
\left({\bar J}/{\bar J}^2\right)_t = \left(J/J^2\right)\otimes_B
\bar B_t \oplus D
$$
If we replace $B\to \bar B\to 0$ by the composition $B\{t_1,\ldots,
t_s\}\to B\to \bar B$, where $s$ is greater than the rank of
$(\Omega^s_{\bar B/\widehat{A}})_t$, we may assume that $D$ is free
over the complement of $V\left(\widehat{\adic} \right)$. This case
is done by Proposition~\ref{prop_good_case}.

\end{proof}

\subsection{Algebraization of a divisor}

\begin{lemma}\label{lemma:invert}
Let $A$ be a ring, $\adic\subseteq A$ an ideal, $B$ a finitely
generated $A$-algebra, and $I$ is an ideal of $B$. Assume that the
ideal $\widehat{I}$ is invertible in $\widehat{B}$, where the hat
means the $\adic$-adic completion. Then there exists an element
$s\in \adic$ such that the ideal $I_{1+s}$ is invertible in
$B_{1+s}$.
\end{lemma}

\begin{proof}
Let us denote the algebra $B_{1+\adic   B}$ by $B'$ and $I_{1+\adic
B}$ by $I'$. Then we have the following sequence of homomorphisms
$B\to B'\to \widehat{B}$. By Theorem~4 of~\cite[Chapter~II,
Section~6]{BourCA}, the ideal $I$ is invertible if and only if $I$
is projective and contains a non-zero divisor.

1) We will prove that $I'$ is invertible. Firstly, we will show that
$I'$ contains a non-zero divisor. Since $\widehat{I}=I'\widehat{B}$,
it is enough to show that, if $I'$ consists of zero divisors only,
then $I'\widehat{B}$ also consists of zero divisors. In this case,
$I'$ belongs to some associated prime $\mathfrak p\subseteq B'$ and
$\mathfrak p =\annihil_{B'}(x)$ for some $x\in B'$. Since
$\widehat{B}$ is flat, we have
$$
\mathfrak p\widehat{B} = \annihil_{B'}(x)\otimes_{B'}\widehat{B}=
\annihil_{\widehat{B}}(x).
$$
In particular, $\mathfrak p\widehat{B}$ consists of zero divisors.

Secondly, we will show that $I'$ is projective. Indeed, the ideal
$\widehat{I}$ is projective, thus, flat. Since $\adic B'$ belongs to
the radical of $B'$, $\widehat{B}$ is faithfully flat. But
$$
\widehat{I} = I'\widehat{B} = I'\otimes_{B'}\widehat{B}.
$$
Hence, $I'$ is flat. Since $I'$ is finitely generated and $B'$ is
Noetherian, $I'$ is projective.

2) Now, we will show that there is some $s\in \adic B$ such that
$I_{1+s}$ is invertible in $B_{1+s}$. Firstly, we will show that
$I_{1+s}$ contains a non-zero divisor for some $s$. If $a'\in I'$ is
a non-zero divisor in $B'$, then, there are an $s_1\in \adic B$ and
$a\in I$ such that $(1+s_1)a'=a$. The annihilator of $a$ in $B$ is
finitely generated and is equal to zero in $B'$. Therefore, there is
an element $s_2\in \adic B$ such that $a$ is not a zero divisor in
$B_{1+s_2}$. Thus element $a\in I_{(1+s_1)(1+s_2)}$ is the required
non-zero divisor.

Secondly, we will show that, for some $s\in \adic B$, the ideal
$I_{1+s}$ is projective $B_{1+s}$-module. By
Lemma~\ref{lemma:H_P_proj}, it is enough to show that
$(B/H_{I})_{1+s}=0$ for some $s\in \adic B$. We can represent $I$ as
follows
$$
0\to K\to F\to I \to 0
$$
where $F$ is a free finitely generated $A$-module. Then by
Corollary~\ref{cor:H_P_K}, we have
$$
H_I = \annihil_B(\ext^1_B(I,K))
$$
But $I'$ is projective $B'$-module, hence $H_{I'}=B'$. Since our
modules are finitely generated over Noetherian ring, localization
commutes with annihilators and Ext. Thus, $(B/H_I)_{1+\adic B}=0$.
Hence, there is an element $s_3\in \adic B$ such that
$(B/H_I)_{1+s_3}=0$. Now, we should localize by the product
$$
(1+s_1)(1+s_2)(1+s_3).
$$
\end{proof}

\begin{theorem}\label{theorem:alg_div}
Let $A$ be a ring, $\adic=(t)\subseteq A$ a principal ideal, $\bar
B$ is a formally finitely generated $\widehat{A}$-algebra being
formally smooth over the complement of $V(\widehat{\adic})$. Assume
that we are given an invertible ideal $\bar I\subseteq \bar B$ such
that the quotient $\bar B/\bar I$ is formally smooth over the
complement of $V(\widehat{\adic})$. Then there exist a finitely
generated $A$-algebra $B$ and an invertible ideal $I\subseteq B$
such that $\widehat{B}=\bar B$ and $I\widehat{B}=\bar I$. Moreover,
algebras $B$ and $B/I$ are smooth over the complement of $V(\adic)$.
\end{theorem}
\begin{proof}

Applying Theorem~\ref{theorem:main} to the homomorphism $\bar B\to
\bar B/\bar I$, we get the following commutative diagram
$$
\xymatrix{
    &{\bar B}\ar[rr]^{\pi}\ar[dl]_{\varphi}&&{\bar B/\bar I}\ar[r]\ar[dl]_{\psi}&0\\
    {\widehat{D}}\ar[rr]^{\widehat{\phi}}&&{\widehat{ D'}}\ar[r]&0&\\
    {D}\ar[rr]^{\phi}\ar[u]&&{D'}\ar[r]\ar[u]&0&\\
}
$$
where $\pi$ is the quotient map, $D$ and $D'$ are finitely generated
$A$-algebras being smooth over the complement of $V(\adic)$, and the
maps $\varphi$ and $\psi$ are $\widehat{A}$-isomorphisms. Now, we
define $J=\ker \phi$. By Lemma~\ref{lemma:invert}, there is an
element $s\in \adic D$ such that $J_{1+s}$ is invertible in
$D_{1+s}$. Now, the algebra $B=D_{1+s}$ and the ideal $I=J_{1+s}$
satisfy the required conditions.
\end{proof}

\bibliographystyle{plain}
\bibliography{bibl}

\end{document}